\documentclass[a4paper,11pt]{amsart}
\usepackage{amsthm}
\usepackage{amsmath}
\usepackage{amssymb}
\usepackage{amscd}           
\usepackage{url}
\usepackage[all]{xy}               
\usepackage{stmaryrd}              
\usepackage{a4wide}
\usepackage[toc,page]{appendix}
\usepackage[utf8]{inputenc}
\usepackage[english]{babel}
\usepackage{todonotes}
\usepackage{thmtools}
\usepackage[fixlanguage]{babelbib}
\usepackage{tikz-cd}
\selectbiblanguage{english}

\newtheorem{theorem}{Theorem}

\newtheorem{lemma}{Lemma}

\newtheorem{proposition}{Proposition}

\newtheorem{problem}{Problem}
\newtheorem{remark}{Remark}

\newcommand{\Aut}{\mathop{\mathrm{Aut} }\nolimits}
\newcommand{\Gal}{\mathop{\mathrm{Gal} }\nolimits}

\newcommand{\SL}{\mathop{\mathrm{SL} }\nolimits}
\newcommand{\GL}{\mathop{\mathrm{GL} }\nolimits}

\newcommand{\tors}{\mathrm{tors}}

\title[On the Ramification of Quadratic Fields where there is a torsion growth]{On the Ramification of Quadratic Fields where the Torsion of a Rational Elliptic Curve Grows}
\author[M. Pineda]{Miguel Pineda-Martín}
\address{Departamento de \'Algebra, Facultad de Matem\'aticas and IMUS, Universidad de Sevilla. Avda. Reina Mercedes s/n, 41012 Sevilla, Spain.}
\email{miguelpinedamartin@gmail.com}
\thanks{This work was supported by the \emph{Ministerio de Ciencia e Innovaci\'on} under Project PID2020-114613GB-I00 (MCIN/AEI/10.13039/501100011033).}

\date{\today}

\begin{document}

\maketitle

\begin{abstract}
Let $E/\mathbb{Q}$ be a rational elliptic curve whose torsion group grows over a quadratic field $K$. In a previous paper, a relation between the primes of the conductor of $E$ and the primes that divide the discriminant of $K$ was shown. In the present paper, we
go further in this study and we compute the ramification of the field $K$, when a torsion point of prime order $\ell$ is added, in terms of invariants of the curve. Concretely, the conductor, the prime $\ell$ and the Kodaira symbol at the primes of bad reduction.
\end{abstract}

\maketitle

\section{Introduction: The problem}
The following notations and conventions will be used throughout the paper:
\begin{itemize}
    \item We will write $\mathcal{C}_r$ for the cyclic group of order $r$.
    \item As it is customary in the context of elliptic curves, the groups are usually written in additive notation.
    \item Given an elliptic curve $E$ defined over a number field $K$, we will write $E(K)$ for the group of points of $E$ with coordinates on $K$ and $E_\tors(K)$ for its torsion subgroup (including the case $K=\mathbb{Q}$).
    \item We will write $o(P)$ for the order of the point $P$ on the group $E(K)$.
    \item Whenever we consider a quadratic number field written as $K = \mathbb{Q} (\sqrt{d})$ we will assume $d$ is a square-free integer and we will denote by $d_K$ the discriminant of the extension $K|\mathbb{Q}$.
    \item Examples are taken from \cite{lmfdb} and labeled accordingly.
\end{itemize}
The discussion started on \cite{deReyna2026} was motivated by the following question originally stated in \cite[Problem 2]{GonTorI}
\begin{problem}\label{problem:1}
Is there a precise (and easy) description of which are the possible quadratic extensions
$K|\mathbb{Q}$ with $E_{\tors}(\mathbb{Q}) \neq E_{\tors}(K)$, ideally in terms of some invariant(s) of the curve?
\end{problem}
In this paper, we provide a satisfactory answer to this question in most of the cases when $K$ is unique, i.e. there's only one quadratic field such that a base change over it makes the torsion grow. The following theorem (proven in \cite[Theorem 2]{GonTorI}) gives the possible groups that can appear as torsion subgroups over $\mathbb{Q}$ and $K$ .

\begin{theorem}\label{thm:1}
With the previous notations, we have the following table:
$$
\renewcommand\arraystretch{1.3}
\begin{array}{|c|l|}
\hline
E_\tors(\mathbb{Q}) & \text{Groups that can appear as } E_\tors(K) \\
\hline
\mathcal{C}_1 & \left\{ \mathcal{C}_1\,,\, \mathcal{C}_3 \,,\, \mathcal{C}_5\,,\, \mathcal{C}_7\,,\, \mathcal{C}_9 \right\} \\
\hline
\mathcal{C}_2 & \left\{ \mathcal{C}_2\,,\, \mathcal{C}_4 \,,\, \mathcal{C}_6\,,\, \mathcal{C}_8\,,\, \mathcal{C}_{10}\,,\, \mathcal{C}_{12}\,,\, \mathcal{C}_{16}\,,\, \mathcal{C}_2 \times \mathcal{C}_{2}\,,\, \mathcal{C}_2 \times \mathcal{C}_{6}\,,\, \mathcal{C}_2 \times \mathcal{C}_{10} \right\} \\ 
\hline
\mathcal{C}_3 & \left\{ \mathcal{C}_3, \; \mathcal{C}_{15}, \; \mathcal{C}_3 \times \mathcal{C}_3 \right\} \\
\hline
\mathcal{C}_4 & \left\{ \mathcal{C}_4 \,,\,\mathcal{C}_8\,,\, \mathcal{C}_{12}\,,\, \mathcal{C}_2 \times \mathcal{C}_{4}\,,\, \mathcal{C}_2 \times \mathcal{C}_{8}\,,\, \mathcal{C}_2 \times \mathcal{C}_{12}\,,\, \mathcal{C}_4 \times \mathcal{C}_4\right\} \\
\hline
\mathcal{C}_5 & \left\{ \mathcal{C}_5, \; \mathcal{C}_{15} \right\} \\
\hline
\mathcal{C}_6& \left\{ \mathcal{C}_6, \; \mathcal{C}_{12}, \; \mathcal{C}_2 \times \mathcal{C}_6, \; \mathcal{C}_3 \times \mathcal{C}_6 \right\} \\
\hline
\mathcal{C}_7 & \left\{ \mathcal{C}_7 \right\} \\
\hline
\mathcal{C}_8 & \left\{ \mathcal{C}_8, \; \mathcal{C}_{16}, \; \mathcal{C}_2 \times \mathcal{C}_8 \right\} \\
\hline
\mathcal{C}_9 & \left\{ \mathcal{C}_9 \right\} \\
\hline
\mathcal{C}_{10} & \left\{ \mathcal{C}_{10}, \; \mathcal{C}_2 \times \mathcal{C}_{10} \right\} \\
\hline
\mathcal{C}_{12} & \left\{ \mathcal{C}_{12}, \; \mathcal{C}_2 \times \mathcal{C}_{12} \right\} \\
\hline
\mathcal{C}_2 \times \mathcal{C}_2 & \left\{ \mathcal{C}_2 \times \mathcal{C}_{2}\,,\, \mathcal{C}_2 \times \mathcal{C}_{4}\,,\, \mathcal{C}_2 \times \mathcal{C}_{6}\,,\, \mathcal{C}_2 \times \mathcal{C}_{8}\,,\,  \mathcal{C}_2 \times \mathcal{C}_{12}\right\} \\
\hline
\mathcal{C}_2 \times \mathcal{C}_4 & \left\{ \mathcal{C}_2 \times \mathcal{C}_4, \; \mathcal{C}_2 \times \mathcal{C}_8, \; \mathcal{C}_4 \times \mathcal{C}_4 \right\} \\
\hline
\mathcal{C}_2 \times \mathcal{C}_6 & \left\{ \mathcal{C}_2 \times \mathcal{C}_6, \mathcal{C}_2 \times \mathcal{C}_{12} \right\} \\
\hline
\mathcal{C}_2 \times \mathcal{C}_8 & \left\{ \mathcal{C}_2 \times \mathcal{C}_8 \right\} \\
\hline
\end{array}
$$
\end{theorem}

Our starting point are the two main theorems of \cite{deReyna2026}
\begin{theorem}\label{thm:2}
    Let $E/\mathbb{Q}$ be an elliptic curve with conductor $N_E$ and $K = \mathbb{Q} (\sqrt{d})$ a quadratic number field with $E_\tors(\mathbb{Q}) \neq E_\tors(K)$. Then if $p \in \mathbb{Z}$ is a prime such that $p|d$, then either $p|N_E$ or $p=3$. 
\end{theorem}

\begin{theorem}\label{thm:3}
    Let $E/\mathbb{Q}$ be an elliptic curve. Assume $K$ is a quadratic number field such that there exists $P\in E(K)[\ell]\setminus E(\mathbb{Q})$, with $\ell\geq 3$ prime.
    \begin{enumerate}
        \item For every prime $p\neq \ell$ that ramifies in $K$, $E$ has additive reduction at $p$, i.e. $p^2|N_E$.
        \item If $p=\ell>3$  ramifies in $K$, $E$ has additive reduction at $p$, i.e. $p^2|N_E$.
    \end{enumerate}
\end{theorem}
Given a rational elliptic curve $E/\mathbb{Q}$ such that $E_{\tors}(\mathbb{Q})\neq E_{\tors}(K)$, where $K = \mathbb{Q}(\sqrt{d})$, we want to compute $d$. These two results allow us to bound the primes that divide $d$. The main results of this paper give conditions on the primes that divide the conductor which imply that they also divide $d$.
\\~\\
Before stating those results let us note that there are some cases where it is easy to compute $K$.
\begin{itemize}
    \item \textbf{$E(\mathbb{Q})[2] = \mathcal{C}_2 $ and $E(K)[2] = \mathcal{C}_2\times \mathcal{C}_2 $}: In this case, the curve over $\mathbb{Q}$ has one and only one non-trivial rational $2$-torsion point. Therefore the $2$-division polynomial has a rational root and two conjugated roots over $K$. Choosing an equation of $E$, and using the definition of the discriminant, is easy to prove that $K = \mathbb{Q}(\sqrt{\Delta})$, where $\Delta$ is the minimal discriminant of $E$.
    \item \textbf{$\mathcal{C}_3\times \mathcal{C}_3$ or $\mathcal{C}_4\times \mathcal{C}_4 \leq E_{\tors}(K)$}: The Weil pairing forces the field embedding $\mathbb{Q}(\zeta_n)\subset \mathbb{Q}(E[n])$. In these cases
    $\mathbb{Q}(\zeta_m)\subset \mathbb{Q}(E[m])\subset K$ with $m\in\{3,4\}$. Then, $K = \mathbb{Q}(\zeta_3) = \mathbb{Q}(\sqrt{-3})$ for $m=3$ and $K = \mathbb{Q}(\zeta_4) = \mathbb{Q}(i)$ for $m= 4$. 
\end{itemize}
Now, as $E_{\tors}(K)\neq E_{\tors}(\mathbb{Q})$ a new torsion point must appear when moving from $\mathbb{Q}$ to $K$. It may happen that some new torsion point $P\in E(K)$ appears with no nontrivial multiple of $P$ belonging to $E(\mathbb{Q})$ (this will be called the \textit{strict} case). Or it may happen that, for each new torsion point $P\in E(K)$, there is a nontrivial multiple $mP$ belonging to $E(\mathbb{Q})$ (this will be the \textit{mixed} case). 
\\~\\
Without loss of generality, we can assume that the point $P$ is of order prime $\ell$ in the strict case and of order $\ell^r$ in the mixed case. The possible values for $\ell$ are
\begin{itemize}
    \item \textbf{Strict Case}: $\ell\in \{2,3,5,7\}$.
    \item \textbf{Mixed Case}: $\ell\in\{2\}$.
\end{itemize}
Note that the strict case $\ell=2$ is one of the trivial ones mentioned above, where $K=\mathbb{Q}(\sqrt{\Delta})$. We won't study the mixed case for $\ell=2$ on this paper. The main results of this paper are:
\begin{theorem}\label{thm:4}
    Let $E/\mathbb{Q}$ be a rational elliptic curve and $\ell$ a prime such that there exists a point $P\in E(K)[\ell]\setminus E(\mathbb{Q})$, where $K = \mathbb{Q}(\sqrt{d})$ and $\ell\geq 3$. Let us consider another prime integer $0<p\in \mathbb{Z}$, we have
    \begin{itemize}
        \item If $\ell\in\{5,7\}$ the following conditions determine the primes that divide $d_K$ in terms of invariants of $E$:
        \begin{itemize}
            \item If  $p\neq \ell$, $p|d_K$ if and only if $E$ has additive reduction at $p$.
            \item If $p=\ell = 5$,  $p|d_K$ if and only if $E$ has additive reduction and the Kodaira symbol at $p$ is different from $II$ and $III$.
            \item If $p=\ell = 7$,  $p|d_K$ if and only if $E$ has additive reduction and the Kodaira symbol at $p$ is different from $II$.
        \end{itemize}
        \item If $p\neq \ell = 3$,   $p|d_K$ if and only if $E$ has additive reduction and the Kodaira symbol at $p$ is different from $IV^*$ and $IV$.
    \end{itemize}
\end{theorem}
\begin{remark}\label{rmk:1}
    Assuming Theorem \ref{thm:4} hypothesis, the statements for $p=\ell = 5$ and $p=\ell=7$ can be restated as 
    \begin{itemize}
        \item If $p=\ell = 5$, $p|d_K$ if, and only if, $E$ has additive reduction and the Kodaira symbol at $p$ is $IV^*$ or $III^*$.
        \item If $p=\ell = 7$, $p|d_K$ if and only if $E$ has additive reduction, and the Kodaira symbol at $p$ is $IV^*$.
    \end{itemize}
    This is proven in Proposition \ref{prop:10} and Proposition \ref{prop:12}.
\end{remark}
This result allows us to determine the primes that divide $d_K$ when $\ell=5,7$. When $\ell = 3$ the result allows us to determine the same for every prime but $3$. This case is more difficult because there are several ways in which a point of order $3$ can appear over a quadratic extension. There are three cases:
\begin{itemize}
    \item $\mathcal{C}_3\times \mathcal{C}_3 \leq E_{\tors}(K)$. We have already seen that $K = \mathbb{Q}(\sqrt{-3})$.
    \item $E_{\tors}(\mathbb{Q})=\mathcal{C}_1$ and $E_{\tors}(K)=\mathcal{C}_9$.
    \item $E_{\tors}(\mathbb{Q})[3]=\mathcal{C}_1$ and $E_{\tors}(K)[3]=\mathcal{C}_3$
\end{itemize}
In the second case, we can determine the primes that divide $|d_K|$ using the following result
\begin{theorem}\label{thm:5}
     Let $E/\mathbb{Q}$ be a rational elliptic curve such that there exists a point $P\in E(K)[9]\setminus E(\mathbb{Q})$ of order $9$, where $K = \mathbb{Q}(\sqrt{d})$. Then, $3|d_K$ if and only if $E$ has additive reduction on $3$ and the Kodaira Symbol at $3$ is not $IV$.
\end{theorem}
The third case is more difficult. For instance, it can happen that a $3$-torsion point appears in two different quadratic extensions, so $K$ would not be unique. We give some discussion in Remark \ref{rmk:5}
\\~\\
The rest of the article is organized as follows:
\begin{itemize}
    \item Section $2$ describes the results on Galois representations that we need for our proofs
    \item Section $3$ presents the changes of variables that are common to our proofs in order to refer to them and simplify the proofs.
    \item Sections $4,5,6$ and $7$ are dedicated to prove Theorems \ref{thm:4} and \ref{thm:5}.
    \begin{enumerate}
        \item Section $4$ is dedicated to prove Theorem \ref{thm:5} and to prove of Theorem \ref{thm:4} for the case $\ell = 3$.
        \item Section $5$ is dedicated to prove auxiliar results that are required to prove Theorem \ref{thm:4} for $\ell = 5,7$.
        \item Section $6$ finishes the proof of Theorem \ref{thm:4} for the case $\ell = 5$.
        \item Section $7$ finishes the proof of Theorem \ref{thm:4} proving the case $\ell = 7$.
    \end{enumerate}
    \item Section $8$ concludes the article and proposes future lines of research to fully resolve Problem \ref{problem:1}.
    \item Section $9$ is an appendix with a presentation of Tate's Algorithm, which is one of the main results used in the proofs. 
\end{itemize}

\section{Images of Inertia}\label{section:2}

A fundamental part of the proof of the previous results is the study of the images of the inertia subgroups at a prime $p$, under the  mod $\ell$ Galois representations attached to an elliptic curve, where $p$ and $\ell$ can be different or equal. We state below the results that we will need.
\begin{proposition}\label{prop:1}
Let $E/\mathbb{Q}$ be an elliptic curve, and let $\ell,p$ be primes. Let
\[
\rho_\ell:\Gal(\mathbb{Q}(E[\ell])/\mathbb{Q})\longrightarrow \Aut(E[\ell])
\]
be the Galois representation attached to the $\ell$-torsion points. Let $I_p:= I(\mathfrak{P}|p)$ be the inertia subgroup associated to primes $\mathfrak{P}|p$, where $\mathfrak{P}$ is a prime ideal of $\mathbb{Q}(E[\ell])/\mathbb{Q}$. Assume $E$ has multiplicative reduction at $p$, then:

\begin{enumerate}
    \item If $\ell\neq p$, the image $\rho_\ell(I_p)$ is either trivial or a group of order $\ell$.
    
    \item If $\ell = p$, then exactly one of the following possibilities occurs:
    
\begin{enumerate}
    \item The wild inertia subgroup $I_{p,w}$ acts trivially on $E[\ell]$, and the image of $I_p$ is a cyclic group of order $\ell - 1$. In a suitable basis, it coincides with the subgroup
    \[
        H_1=\left\{\begin{pmatrix} a & 0\\ 0 & 1\end{pmatrix}: a\in \mathbb{F}_{\ell}^{\times}\right\}.
    \]
    
    \item The wild inertia subgroup $I_{p,w}$ does not act trivially on $E[\ell]$. In this case, the image of $I_{p,w}$ is a cyclic group of order $\ell$, and in a suitable basis it can be represented as
    \[
        \left\{\begin{pmatrix} 1 & b\\ 0 & 1\end{pmatrix}: b\in \mathbb{F}_{\ell}\right\}.
    \]
    The image of $I$ has order $\ell(\ell - 1)$ and can be represented as
    \[
        H_2=\left\{\begin{pmatrix} a & b\\ 0 & 1\end{pmatrix}: a\in \mathbb{F}_{\ell}^{\times},\ b\in \mathbb{F}_{\ell}\right\}.
    \]
\end{enumerate}
\end{enumerate}
\end{proposition}
\begin{proof}
    See \cite[Section 1.12]{Serre}. 
\end{proof}

In the case of additive reduction the classification is more complex. It suffices to present it over local fields. Let $\ell,p$ be primes (not necessarily distinct), and let $E/\mathbb{Q}_p$ be an elliptic curve with additive reduction at $p$. There always exists a finite extension $K'/\mathbb{Q}_p$ such that the reduction of $E$ over $K'$ is either good or multiplicative (see \cite[Proposition 5.4]{Silverman}). Moreover, the following criterion allows us to determine the type of potential reduction.

\begin{proposition}\cite[Proposition 5.5]{Silverman}
Let $E/K$ be an elliptic curve defined over a local field $K$, with ring of integers $R$. Then $E$ has potentially good reduction if and only if its $j$-invariant satisfies $j(E)\in R$.
\end{proposition}

In this context, there is a nontrivial group $\Phi_p$ that measures how far the curve $E$ is from having multiplicative or good reduction at $p$. It can be defined as
\[
\Phi_p = \Gal(L/\mathbb{Q}_p^{\mathrm{nr}}),
\]
where $\mathbb{Q}_p^{\mathrm{nr}}$ is the maximal non-ramified extension of $\mathbb{Q}_p$ and  $L$ is the minimal extension of $\mathbb{Q}_p^{\mathrm{nr}}$ over which the type of reduction changes. 
\\~\\
Let us assume that $E/\mathbb{Q}_p$ has potential good reduction. In this case, the group $\Phi_p$ is described by the following proposition.

\begin{proposition}\label{prop:3}
     Let $E/\mathbb{Q}_p$ be an elliptic curve. Let $L|\mathbb{Q}_p^{\mathrm{nr}}$ be the extension previously defined and $\Phi_p:=\Gal(L/\mathbb{Q}_p^{\mathrm{nr}})$. Then for all $p\neq 2,3$ we have:
     {\small
     \begin{enumerate}
         \item $\#\Phi_p = 2\iff \text{ $E$ has Kodaira symbol $I_0^*$} \iff v_p(\Delta) \equiv 6 \mod 12$ 
         \item $\#\Phi_p = 3\iff \text{ $E$ has Kodaira symbol $IV$ or $IV^*$} \iff v_p(\Delta) \equiv 4 \text{ or } 8 \mod 12$
         \item $\#\Phi_p = 4\iff \text{ $E$ has Kodaira symbol $III$ or $III^*$} \iff v_p(\Delta) \equiv 3 \text{ or } 9 \mod 12$
         \item $\#\Phi_p = 6\iff \text{ $E$ has Kodaira symbol $II$ or $II^*$} \iff v_p(\Delta) \equiv 2 \text{ or } 10 \mod 12$
     \end{enumerate}}
If $p=2,3$, then 
\begin{enumerate}
         \item If $p = 3$,  $\Phi_p$ is a cyclic group of order $2,3,4,5,6$ or a non-abelian semidirect product of a cyclic group of order $4$ with a normal subgroup of order $3$.
         \item If $p = 2$, $\Phi_p$ is isomorphic to a subgroup of $\SL_2(\mathbb{F}_3)$ and its order is $2,3,4,6,8$ o $24$. Moreover, $\#\Phi_p = 3$ if, and only if, the Kodaira symbol of $E$ is $IV$ or $IV^*$.
     \end{enumerate}
\end{proposition}
\begin{proof}
    See \cite[pág. 312]{Serre} for every statement except the last statement of the case $p=2$. Alain Kraus conducted an exhaustive study of the cases $p=2,3$ in \cite{Kraus1990}, but we only need \cite[Théorème 2]{Kraus1990}.
\end{proof}

 According to \cite{SerreTate} and \cite[p.~311, Section 5.6 a)]{Serre}, the action of $I_p$ on $E[\ell]$ is described by the following diagram:
\[
\begin{tikzcd}
I_p \arrow[r, "\varphi_\ell"] \arrow[d, two heads] & {\Aut(E[\ell])} \\
\Phi_p \arrow[ru]                         &       
\end{tikzcd}
\]
where $\Phi_p$ is a finite quotient of inertia whose action is injective if $\ell\geq 3$ and $p\neq \ell$. So, with this hypothesis, the image of inertia is isomorphic to $\Phi_p$.

Let us assume now that $E$ has potential multiplicative reduction. In this case, the curve $E/\mathbb{Q}_p$ is isomorphic to the Tate curve with the same $j$-invariant over a quadratic extension $L/\mathbb{Q}_p$, which is ramified at $p$ because the reduction at $p$ is additive, see \cite[V.5.3]{Silverman1994AdvancedTI}. The group $\Phi_p = \Gal(L/\mathbb{Q}_p)$ is a quotient of the inertia group $I_p$, although injectivity need not hold in this case. By properties of the Tate curve, $E$ has multiplicative reduction over $L$, and therefore its inertia is known.

\section{Variable changes and Tate normal form}\label{section:3}

The study of the images of the Galois representations is not enough to conclude the proofs of the main results. We will use a more computational approach where we will use explicit models for the curves along with standard results on the torsion growth.
\\~\\
The fact that we are going to do explicit computations with models of the elliptic curve causes that very similar arguments and changes of variables are repeated in different proofs. In this section, we gather them together.
\\~\\
We begin with an elliptic curve $E/\mathbb{Q}$ such that there exists a quadratic field $K=\mathbb{Q}(\sqrt{d})|\mathbb{Q}$ for which there exists a point $P\in E(K)[n]\setminus E(\mathbb{Q})$ which is a torsion point of order $n$ with $n\geq 3$ and odd. Suppose moreover that $E(\mathbb{Q})[n] = \{0\}$. We start with a minimal model 
$$
E_1 : y^2 + a_1xy + a_3y = x^3 + a_2x^2 + a_4x + a_6
$$
with integer coefficients such that $a_1,a_3\in\{0,1\}, \quad a_2\in\{-1,0,1\}$ and $\Delta_1$ denotes its discriminant. We perform the change 
$$
x = x',\quad y = y' -\frac{1}{2}\left( a_1x' + a_3 \right).
$$
which leads to the equation
$$
 E': y^2= x^3 + \frac{b_2}{4}x^2 + \frac{b_4}{2}x + \frac{b_6}{4}.
$$
where the formulas for $b_i$ can be found in \cite[Ch. III]{Silverman}. It is easy to see that the discriminant of this equation is still $\Delta$. To obtain an equation with integer coefficients, we perform the change
$$
x' = 4x,\quad y' = 8y
$$
and obtain the equation (in $\mathbb{Z}[x,y]$)
$$
E_2: y^2 = x^3 + b_2x^2 + 8b_4x + 16 b_6.
$$
Let $\Delta_2$ denote the discriminant of the previous equation. We have
$$
2^{12}\Delta_1 = \Delta_2.
$$
Once we reach this point, by (\cite[Corollary 4]{GonTorI}) we have 
$$
    E(K)[n]=  E(\mathbb{Q})[n]\times E^d(\mathbb{Q})[n].
$$
where $E^d$ is the twist
$$
E^d: dy^2= x^3+ b_2x^2 + 8b_4x + 16b_6.
$$
Therefore, $E(K)[n]= E^d(\mathbb{Q})[n]$. Hence the curve $E^d$ has a point of order $n$. It is easy to bring the curve $E^d$ to Weierstrass form
$$
 E_3 : y^2=x^3 +db_2x^2 + 8d^2b_4x + 16d^3b_6.
$$
by a change defined over $\mathbb{Q}$. The change to pass from $E_2$ to $E_3$ is 
$$
    x' = dx,\quad y'= d\sqrt{d} y.
$$
Therefore, the relation between the discriminants is $\Delta_1 2^{12} d^6 = \Delta_3 $. As we saw before, the curve $E_3$ has a point of order $n$, which we denote by $Q= (x_1,y_1)$. Note that, by the Nagell–Lutz theorem, $x_1,y_1\in\mathbb{Z}$. We now bring the curve to Tate normal form.
\\~\\
Any change of variables preserving the Weierstrass form must be of the form
$$
x = u^2x' + r,\quad y = u^3y' + su^2x' + t
$$
The only parameter that changes the discriminant is $u$, so we impose $u=1$. Moreover, we want to send $Q$ to $(0,0)$. Since
$$
(x_1,y_1) \longmapsto \big( x_1-r, y_1-s(x_1-r) -t \big),
$$ 
we need $r= x_1$ and $t=y_1$. This cancels the constant term of the equation. Now we choose $s$ so that the coefficient of $x$ is $0$, which gives the following equation
$$
0 = 8d^2b_4 + 2db_2x_1 + 3x_1^2 -2sy_1.
$$
Thus we obtain
\begin{equation}\label{eq:1}
s = \frac{8d^2b_4 + 2 db_2x_1 + 3x_1^2}{2y_1}.
\end{equation}

Note that $y_1\neq 0$ because $n$ is odd. We obtain the following form
$$
E_4: y^2 + \overline{a}_1xy+ \overline{a}_3y = x^3 + \overline{a}_2x^2,
$$
with
$$
\def\arraystretch{1.5}
\left\{ \begin{array}{lcl} 
\overline{a}_1 & = & 2s\\ 
\overline{a}_2 & = & -s^2+ 3x_1 +db_2 \\ 
\overline{a}_3 & = & 2y_1 \\ 
\Delta_4 & = & 2^{12} d^6\Delta_1 
\end{array} \right.
$$
If $n=3$ the point $(0,0)$ of this equation is a $3$-torsion point. This is equivalent to saying that the curve has a contact of order $3$ with the line $y=0$, which implies that $\overline{a}_2 = 0$. So the equation that we get is
$$
E_4: y^2 + \overline{a}_1xy+ \overline{a}_3y = x^3.  
$$
If $n\geq 5$, we can reduce the curve to obtain the Tate normal form. In order to do so, it only remains to make the coefficients of $y$ and $x^2$ equal. For this purpose we pass to the equation
$$
E_5: y^2 + \Tilde{a}_1xy+ \Tilde{a}_3y = x^3 + \Tilde{a}_2x^2,
$$
by means of the change of variables
$$
x \longmapsto \left(\frac{\overline{a}_3}{\overline{a}_2}\right)^2 x, \qquad y \longmapsto \left(\frac{\overline{a}_3}{\overline{a}_2}\right)^3 y.
$$
In this way we obtain
$$
\Tilde{a}_1 = \dfrac{\overline{a}_1\overline{a}_2}{\overline{a}_3}, \qquad 
\Tilde{a}_2 = \Tilde{a}_3 = \dfrac{\overline{a}_2^3}{\overline{a}_3^2}.
$$

\

If we denote $b := -\Tilde{a}_2= -\Tilde{a}_3$ and $c:= 1- \Tilde{a}_1$, we obtain the Tate normal form
$$
\mathcal{T}_{b,c}:  y^2+ (1-c)xy - by = x^3 - bx^2
$$
with the following relation between discriminants
\begin{equation}\label{eq:2}
\Delta_{b,c}\overline{a}_3^4 = 2^{12}b^4 d^6\Delta_1.
\end{equation}
Moreover, for values $n\leq 9$, we will use that these curves can be parametrized by a single parameter. In particular we will use the following parametrizations
\begin{itemize}
    \item $n = 4$. In this case, $c = 0$ and $\Delta_{b,c}= b^4(1+ 16b)$
    \item $n = 5$. In this case, $b = c$ and $\Delta_{b,c}= b^5(b^2 - 11b -1)$.
    \item $n = 7$. In this case, there exists $t\in\mathbb{Q}$ such that $b = t^2(t-1)$ and $c = t(t-1)$. Moreover, $\Delta_{b,c} = t^7(t-1)^7(t^3-8t^2 + 5t + 1)$.
    \item $n = 8$. In this case, there exists $t\in\mathbb{Q}$ such that $b = (t-1)(2t -1)$ and $c = b/t$. Moreover, $\Delta_{b,c} =  t^{-4}(1-2t)^4(t-1)^8(8(t-1)t + 1)$
    \item $n = 9$. In this case, there exists $t\in\mathbb{Q}$ such that $b =t^5 - 2t^4 + 2t^3 - t^2$ and $c =  t^3 - t^2$. Moreover, $\Delta_{b,c} = t^9(t-1)^9(t^2-t+1)^3(t^3-6t^2 + 3t + 1)$.
\end{itemize}
The parametrizations can be found in \cite[Chapter 4, section 4]{Husemoller}

\begin{remark}\label{rmk:2}
Note that if in the original minimal model $a_1=a_3= 0$, then the equations of $E_1$ and $E'$ coincide and therefore the coefficients are integers. Hence there is no need to perform the change to obtain $E_2$. In this case we may proceed as before and the equations we obtain are the following.
\\~\\
The equation of $E_3$ is
$$
E_3: y^2 = x^3+ da_2x^2 + d^2a_4x + d^3a_6
$$
and it has an integral $n$-torsion point $Q= (x_1, y_1)$. Following the previous procedure we arrive at the curve 
$$
E_4: y^2 + \overline{a}_1xy+ \overline{a}_3y = x^3 + \overline{a}_2x^2,
$$
with
$$
\def\arraystretch{1.5}
\left\{ \begin{array}{lcl} 
s &=& \displaystyle   \frac{d^2a_4 + 2da_2x_1 + 3x_1^2}{2y_1} \\
\overline{a}_1 & = & 2s\\ 
\overline{a}_2 & = & -s^2+ 3x_1 +da_2 \\ 
\overline{a}_3 & = & 2y_1 \\ 
\Delta_4 & = &  d^6\Delta_1 
\end{array} \right.
$$
The rest is the same as in the previous case, except that the final relation between discriminants is 
\begin{equation}\label{eq:3}
  \Delta_{b,c}\overline{a}_3^4 = b^4 d^6\Delta_1.  
\end{equation}
\end{remark}

\begin{remark}\label{rmk:3}
\begin{itemize}
    \item Note that in the procedure we have just described, if the equation $E_4$ has $p$-integral coefficients and $2\neq p\nmid d$, then the equation of $E_4$ is minimal, since the original equation was minimal, the twist was taken over a field in which $p$ does not ramify, and the valuation of the discriminant is the same. Moreover, if $p= 2$ and we are in the hypotheses of the previous remark, we can argue in the same way because the change that alters the $2$-adic valuation of the discriminant was not performed.
    \item Similarly, if $p\nmid d_K$ then the reduction type of the curve and of the twist is the same (see \cite[VII.5.4.(a)]{Silverman}).
\end{itemize}
\end{remark}
\begin{remark}\label{rmk:4}
  If $E$ has additive reduction at a prime $p$ such that $p\nmid n$ and we call $Q$ the point of $n$-torsion of the twisted curve, then $Q$ reduces to a singular point.
  \\~\\
  Indeed, suppose that its reduction is not singular, that is,
$$
\Tilde{Q} \in \Tilde{E}_{ns}.
$$
After reducing the curve, this point must either have order a divisor of $n$. Since the reduction of the curve at $p$ is additive, the group
$$
\Tilde{E}_{ns} \cong \overline{\mathbb{F}_p}^{+},
$$
which has no $d$-torsion for $d$ a proper divisor of $n$. Therefore, the point $Q$ must reduce to the identity element. However, the kernel of reduction does not contain points of order coprime to $p$, see \cite[VII.Proposition 3.1(a)]{Silverman}. Hence the point $Q$ must reduce to a singular point.
\end{remark}
After making these changes of variables, the proofs will consist of appropriately applying Tate's Algorithm to one or several of the previous models. Every reference in what follows to a concrete step of Tate's Algorithm will refer to its presentation of the Appendix of this article. The following result will also be very useful during the proofs in order to bound the valuation of the minimal discriminant.
\begin{theorem}[Ogg's Formula]\label{thm:6}
    Let $E/K$ be an elliptic curve defined over a local field $K$, with ring of integers $R$ and residual field $k$. Consider 
    \begin{itemize}
        \item $v_K(\Delta)$.  The valuation of the minimal discriminant of $E/K$.
        \item $f(E/K)$. The exponent of the conductor of $E/K$.
        \item $m(E/K)$. The number of components, defined over $\overline{k}$ and counted without multiplicity, on the special fiber of a minimal proper regular model of $E$ over $R$.
    \end{itemize}
    Then
    $$
    v_K(\Delta) = f(E/K) + m(E/K) -1
    $$
\end{theorem}
\begin{proof}
    See \cite[IV.11.1] {Silverman1994AdvancedTI}.
\end{proof}
Now we proceed to prove Theorem \ref{thm:4} and Theorem \ref{thm:5}. Let $E$ be an elliptic curve such that there exists $P\in E(K)[\ell]\setminus E(\mathbb{Q})$, where $K=\mathbb{Q}(\sqrt{d})$ and $\ell$ is a prime number. We will separate the proofs by the value of $\ell\in\{3,5,7\}$.

\section{Prime $\ell = 3$}
We start by proving the statement of Theorem \ref{thm:4} where $\ell =3$. The following result proves it in one direction.
\begin{proposition}\label{prop:4}
Let $p\neq 3$ be a prime and $E/\mathbb{Q}$ be an elliptic curve such that there exists a quadratic field $K$ and a $3$-torsion point $P\in E(K)[3]\setminus E(\mathbb{Q})$. Suppose moreover that $3|d_K$. Then $E$ has additive reduction at $p$ with Kodaira symbol different from $IV$ or $IV^*$.
\end{proposition}

\begin{proof}
First, note that the first statement of Theorem \ref{thm:3} already shows that $E$ has additive reduction at $p$. Let $\mathfrak{P}|p$ be a prime ideal of $\mathbb{Q}(E[3])$. Let $I(\mathfrak{P}|p)$ be the inertia group and let $L = \mathbb{Q}(E[3])^{I(\mathfrak{P}|p)}$ be the maximal unramified subextension.
We have the following diagram
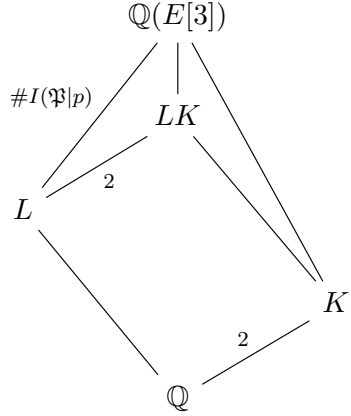
\[\begin{tikzcd}
	& {\mathbb{Q}(E[3])} \\
	& LK \\
	L \\
	&& K \\
	& {\mathbb{Q}}
	\arrow["{\# I(\mathfrak{P}|p)}"', no head, from=1-2, to=3-1]
	\arrow[no head, from=2-2, to=1-2]
	\arrow[no head, from=2-2, to=4-3]
	\arrow["2"', no head, from=3-1, to=2-2]
	\arrow[no head, from=3-1, to=5-2]
	\arrow[no head, from=4-3, to=1-2]
	\arrow["2", no head, from=5-2, to=4-3]
\end{tikzcd}\]

We proceed by contradiction. Let us assume that the curve has Kodaira symbol $IV$ or $IV^*$ at $p\neq 3$, the $j$-invariant of $E$ is integral (see \cite[IV.9.Table 4.1]{Silverman1994AdvancedTI}). This implies that $E$ has potentially good reduction, hence $I(\mathfrak{P}|p) = \Phi_p$. The diagram shows that $2 \mid \#\Phi_p$. However, since the Kodaira symbol of $E$ at $p$ is $IV$ or $IV^*$, Proposition \ref{prop:3} implies that $\#\Phi_p = 3$, which yields a contradiction.
\end{proof}
In order to prove the reciprocal we need to separate the proof in the case $p\neq 2$ and $p=2$. The following lemmas complete the proof.

\begin{lemma}\label{lemma:1}
Let $E/\mathbb{Q}$ be an elliptic curve such that there exists a quadratic field $K$ and a point $P\in E(K)[3]\setminus E(\mathbb{Q})$. If the curve $E$ has additive reduction at a prime $p\neq 2,3$ and the Kodaira symbol at $p$ is neither $IV$ nor $IV^*$, then $p \mid d_K$.
\end{lemma}

\begin{proof} 
Let $p\geq 5$ be a prime and suppose that $E$ has additive reduction at $p$ and that the Kodaira symbol is neither $IV$ nor $IV^*$. We argue by contradiction and assume that $p$ does not ramify in $K$. We begin with a minimal equation at $p$
$$
E_1: y^2+a_1xy + a_3y = x^3+a_2x^2+a_4x+a_6.
$$
Following the procedure of Section \ref{section:3} for $n=3$, we obtain a quadratic twist over $K$
$$
E_4:y^2+\overline{a}_1xy + \overline{a}_3y = x^3
$$
Since $\overline{a}_2=0$, it follows that $s$ must be integral. Therefore, the coefficients of $E_4$ are integral and, since the twist we have made is unramified, its equation must be minimal and have the same type of reduction as $E$ (see Remark \ref{rmk:3}).
\\~\\
As $E$ has additive reduction at $p\neq 3$, $(0,0)$ reduces to a singular point (see Remark \ref{rmk:4}). Applying steps 1 and 2 of Tate's algorithm to $E_4$ we obtain
$$
0<v_p(\overline{a}_1^2+4\overline{a}_2) = 2v_p(\overline{a}_1) = 2v_p(s).
$$
Since the equation of $E_4$ is minimal, $v_p(\overline{a}_3) = v_p(y_1)\in\{0,1,2\}$. Otherwise, $v_p(\overline{a}_i)\geq i$ for $i\in\{1,2,3,4,6\}$. On the other hand, we have
$$
v_p(\Delta_1) = v_p(\Delta_4) = 3v_p(y_1) + v_p(4s^3-27y_1).
$$

If $v_p(y_1) = 0$, we obtain $v_p(\Delta_1)=0$, which contradicts the fact that the curve has additive reduction. If $v_p(y_1)=1$, then $v_p(\Delta_1)=4$ and, since $p\neq 2,3$, Tate's algorithm implies that the Kodaira symbol of $E$ is $IV$ (see \cite[IV.Table 4.1]{Silverman1994AdvancedTI}). If $v_p(y_1)=2$, then $v_p(\Delta_1)=8$. Then the Kodaira symbol of $E$ is $IV^*$ or $I_2^*$. However, $I_2^*$ cannot happen. If we apply Tate's Algorithm to the model $E_4$, we reach step $6$ without making any change of variables. The polynomial $P(T)$, defined at this step, is $P(T) = T^3$ because of the shape of the model $E_4$. This implies that the algorithm cannot stop at the step $7$, which implies that the Kodaira symbol of the curve cannot be $I_n^*$. 
\end{proof}
We now prove the case $p=2$.
\begin{lemma}
Let $E/\mathbb{Q}$ be an elliptic curve such that there exists a quadratic field $K = \mathbb{Q}(\sqrt{d})$ and a point $P\in E(K)[3]\setminus E(\mathbb{Q})$. Suppose that $E$ has additive reduction at $2$ and that its Kodaira symbol at $2$ is neither $IV$ nor $IV^*$. Then $2$ ramifies in $K$.
\end{lemma}
\begin{proof} 
Arguing by contradiction, suppose that $2$ does not ramify in $K$. In particular $v_2(d)=0$. We start with a minimal model of $E$ at $2$,
$$
E_1 : y^2 + a_1xy + a_3y = x^3 + a_2x^2+a_4x +a_6.
$$

Following the procedure explained at the beginning of the previous proof, we obtain the quadratic twist
$$
y^2+\overline{a}_1xy + \overline{a}_3y = x^3 + \overline{a_2}x^2,
$$
where $\overline{a}_2 = 0$, which implies that the equation has integral coefficients. In this case, the relation between the valuations of the discriminants is
$$
12+ v_2(\Delta_1) = v_2(\Delta_4) = 4 + 3v_2(y_1) + v_2(4s^3-27y_1),
$$ 
It is easy to use this equation to prove that $v_2(y_1)\geq 2$. Thus $v_2(\overline{a}_i)\geq i$, and therefore $E_4$ is not minimal. Performing a change of variables to decrease the valuation, we obtain the curve
$$
E_4':y^2+\overline{a}_{1,1}xy + \overline{a}_{3,3}y = x^3,
$$
with $\overline{a}_{i,j} = \dfrac{\overline{a}_i}{2^j}$. Note that this model is minimal because the valuation of its discriminant is equal to $v(\Delta_1)$ and $E_1$ is a minimal model.
\\~\\
We apply Tate's algorithm to this model, of which we only use its shape, the fact that its coefficients are integral, and that the reduction is additive. The latter follows from the fact that $E$ and its twist have the same type of reduction (see Remark \ref{rmk:3}). If we apply Tate's algorithm, all the conditions are satisfied until step $5$, where we have two possibilities: either $v_2(\overline{a}_{3,3})\geq 2$ or the Kodaira symbol is $IV$. We assume we are in the first case because the second already gives the desired conclusion. Continuing the algorithm, we reach step $8$, where the curve has Kodaira symbol $IV^*$. If the algorithm do not stop at step $8$, it will continue running and, because of the shape of $E_4'$, we obtain that the curve is not minimal, which is a contradiction.
\end{proof}

This concludes the proof of the case $\ell= 3$ of Theorem \ref{thm:4}. We now prove Theorem \ref{thm:5}. First, we need the following lemma that shows that if there appears a point of order $9$ over a quadratic field ramified at $3$ the original curve has additive reduction at $3$.

\begin{proposition}\label{prop:5}
Let $E/\mathbb{Q}$ be an elliptic curve such that there exists a quadratic field $K = \mathbb{Q}(\sqrt{d})$, with discriminant $d_K$, for which there exists $P\in E(K)\setminus E(\mathbb{Q})$ of order $9$. Then, if $p\mid d_K$, the curve $E$ has additive reduction at $p$.
\end{proposition}
\begin{proof}
By Theorem \ref{thm:1}, if such a point $P$ of order $9$ exists, then $[3]P\in E(K)\setminus E(\mathbb{Q})$. So  Theorem \ref{thm:3} gives us the result for $p\neq 3$. Therefore, it suffices to prove the result for $p = 3$. We argue by contradiction and assume that $E$ does not have additive reduction. Because of \cite[Proposition 7]{deReyna2026} $E$ has multiplicative reduction at $3$. We can always find a minimal equation for $E$
$$
E_1 : y^2 + a_1xy + a_3y = x^3 + a_2x^2 + a_4x + a_6
$$
with integer coefficients such that $a_1,a_3\in\{0,1\}$ and $a_2\in\{-1,0,1\}$. After performing the procedure explained in Section \ref{section:3}, we obtain the twist of $E_1$ in Tate normal form
$$
\mathcal{T}_{b,c}:  y^2+ (1-c)xy - by = x^3 - bx^2
$$
with $b :=- \dfrac{\overline{a}_2^3}{\overline{a}_3^2}$ and $c:= 1- \dfrac{\overline{a}_1\overline{a}_2}{\overline{a}_3}$. Moreover, since the point $(0,0)$ is of $9$-torsion, there exists $t\in\mathbb{Q}$ such that
$$
c = t^2(t - 1), \qquad b =t^2(t-1)( t^2 - t + 1)
$$
and
$$
\Delta_{b,c} = t^9(t-1)^9(t^2-t+1)^3(t^3-6t^2 + 3t + 1).
$$
Substituting into equation \eqref{eq:2} of Section \ref{section:3} we obtain
\begin{equation}\label{eq:4}
t(t-1)^5(t^3-6t^2 + 3t + 1)\overline{a}_3^4 = 2^{12} d^6\Delta_1(t^2-t+1).
\end{equation}

Note that, following the procedure of Section \ref{section:3}, the change performed from $E_1$ to $E_4$ modifies the coefficient $c_4$ (see \cite[p. 42]{Silverman}) according to \cite[III Table 3.1]{Silverman}. It is easy to check that $\overline{c}_4 = 2^4d^2c_4$, where $\overline{c}_4$ is the coefficient corresponding to the equation $E_4$. In particular, this implies that if the coefficients of $E_4$ are $3$-integral, then the equation $E_4$ is minimal by \cite[VII Remark 1.1]{Silverman}. We continue the proof by splitting into cases according to the sign of the valuation of $t$.
\\~\\
\noindent \fbox{$v_3(t)< 0$}
Taking the $3$-adic valuation in equation \eqref{eq:4}, we obtain
$$
4v_3(y_1)= 6 + v_3(\Delta_1)- 7v_3(t)\geq 14.
$$
This implies that $v_3(y_1)\geq 4$. We also have $v_3(\overline{a}_1) = v_3(s)$. If $v_3(s)\leq 0$, it is easy to check that $v_3(\overline{a}_2) = 2v_3(s)$. Taking the $3$-adic valuation in the definition of $b$, we obtain on the one hand
\begin{equation}\label{eq:5}
5v_3(t) = 2(3v_3(s)- v_3(y_1))
\end{equation}
and, on the other hand, since $v_3(c)= v_3\left(1-\dfrac{\overline{a}_1\overline{a}_2}{\overline{a}_3}\right)<0$, we have
\begin{equation}\label{eq:6}
3v_3(t) = 3v_3(s) - v_3(y_1).
\end{equation}
Since $v_3(t)<0$, equations \eqref{eq:5} and \eqref{eq:6} imply that $v_3(s) = v_3(y_1) = 0$, which is a contradiction. Hence $v_3(s)>0$, which implies that
$$
v_3(\overline{a}_2) = v_3(-s^2 + 3x_1+ db_2) = 1.
$$
Indeed, $v_3(b_2) = 0$ and $v_3(x_1)>0$. The first follows from the fact that $ 0 =v_3(c_4) = v_3(b_2-24b_4)$ and the second from
$$
0<2v_3(y_1) = v_3(x_1^3 +db_2x_1^2 + 8d^2b_4x_1 + 16d^3b_6).
$$
Taking again the $3$-adic valuation of $b$ and $c$, we arrive at the equations
$$
\left\{
\begin{aligned}
5v_3(t) &=3 - 2v_3(y_1)\\
3v_3(t) &= v_3(s) + 1-v_3(y_1)
\end{aligned}
\right.
$$
Multiplying the second by $2$ and subtracting the first, we obtain
$$
0>v_3(t) = 2v_3(s)-1\geq 1.
$$
\noindent \fbox{$v_3(t)> 0$}
\\~\\
Taking the $3$-adic valuation in the definition of $b$, we obtain on the one hand
\begin{equation}\label{eq:7}
0<2v_3(t) = 3v_3(\overline{a}_2)- 2v_3(y_1),
\end{equation} 
which implies that $v_3(\overline{a}_2)>0$, and this in turn implies that $v_3(s)>0$. On the other hand, since $0<v_3(c)= v_3\left(1-\dfrac{\overline{a}_1\overline{a}_2}{\overline{a}_3}\right)$, we have
\begin{equation}\label{eq:8}
0 = v_3(\overline{a}_1) + v_3(\overline{a}_2) - v_3(y_1).
\end{equation}
Hence
$$
v_3(y_1) = v_3(\overline{a}_1) + v_3(\overline{a}_2)\geq 2.
$$
Moreover,
$$
v_3(\overline{a}_2) = v_3(-s^2 + 3x_1+ db_2) = 1
$$
because, as in the previous case, $v_3(b_2) = 0$ and $v_3(x_1)>0$. Substituting into equation \eqref{eq:7}, we obtain
$$
0<2v_3(t) = 3- 2v_3(y_1)\leq -1,
$$
a contradiction.
\\~\\
\noindent \fbox{$v_3(t)= 0$}
\\~\\
We divide this case into two subcases.
\\~\\
\noindent \fbox{\textbf{Case I} $t\equiv 1\mod 3$}
\\~\\
In this case, $v_3(t-1)>0$. Taking the $3$-adic valuation in the definition of $b$, we obtain 
\begin{equation}\label{eq:9}
0<v_3(t-1) = 3v_3(\overline{a}_2)- 2v_3(y_1).
\end{equation}
This implies that $v_3(\overline{a}_2) >0$  and $v_3(s)>0$. Moreover, $v_3(c) = v_3(t-1)>0$, hence
$$
v_3(y_1) = v_3(\overline{a}_1) + v_3(\overline{a}_2)\geq 2.
$$
This implies that $v_3(\overline{a}_2) = 1$. Indeed, $v_3(\overline{a}_2) = v_3(-s^2 + 3x_1+ db_2), v_3(b_2) = 0$ and $v_3(x_1)>0$. The two last statements follow from $ 0 =v_3(c_4) = v_3(b_2-24b_4)$ and 
$$
0<2v_3(y_1) = v_3(x_1^3 +db_2x_1^2 + 8d^2b_4x_1 + 16d^3b_6),
$$
respectively. Now, substituting $v_3(\overline{a}_2) = 1$ into equation \eqref{eq:9}, we obtain 
$$
0<v_3(t-1) = 3- 2v_3(y_1)\leq -1,
$$
a contradiction.
\\~\\
\noindent \fbox{\textbf{Case II} $t\equiv -1\mod 3$}
\\~\\
In this case, $t =  -1 + 3k$. Substituting into the polynomials we obtain
\begin{eqnarray*}
t^3-6t^2+3t+1 &=& 27k^3 - 81k^2 + 54k - 9\\
t^2-t+1 &=& 9k^2 - 9k + 3.
\end{eqnarray*}
Thus
$$
v_3(t^3-6t^2+3t+1) = 2, \qquad v_3(t^2-t+1) = 1.
$$
Taking the $3$-adic valuation of $b$, we obtain
$$
1 = 3v_3(\overline{a}_2) - 2v_3(y_1).
$$
Hence $v_3(\overline{a}_2),v_3(y_1)>0$, which implies that $v_3(s)>0$. The discriminant equation \eqref{eq:4}  in this case is
$$
4v_3(y_1) = v_3(\Delta_1) + 5.
$$
Thus $v_3(y_1)\geq 2$. However, it is easy to see that the equation
$$
1 = 3v_3(\overline{a}_2) - 2v_3(y_1)
$$
implies that $v_3(\overline{a}_2)\geq 2$ and $v_3(y_1)\geq 3$. Since $v_3(s)>0$, this contradicts the fact that the equation of $E_4$ is minimal.
\end{proof}

\addtocounter{theorem}{-1}
\begin{theorem}
     Let $E/\mathbb{Q}$ be a rational elliptic curve such that there exists a point $P\in E(K)[9]\setminus E(\mathbb{Q})$ of order $9$, where $K = \mathbb{Q}(\sqrt{d})$. Then, $3|d_K$ if and only if $E$ has additive reduction at $3$ and the Kodaira Symbol at $3$ is not $IV$.
\end{theorem}
\begin{proof}
    Fix a minimal model at $3$
    $$
    y^2 + a_1xy + a_3y = x^3+ a_2x^2+ a_4x+a_6.
    $$
    Applying the procedure of Section \ref{section:3}, we reach to the following curve in Tate normal form, which is a twist of the original curve over the field $K$,
    $$
    \mathcal{T}_{b,c}: y^2 + (1-c)xy -by = x^3-bx^2
    $$
    with $b :=- \dfrac{\overline{a}_2^3}{\overline{a}_3^2}$ and $c:= 1- \dfrac{\overline{a}_1\overline{a}_2}{\overline{a}_3}$.
     By hypothesis, the point $(0,0)$ on the curve in Tate normal form is a $9$-torsion point. Therefore, there exists $t\in\mathbb{Q}$ such that $c = t^3 - t^2, b = t^2(t-1)(t^2-t+1)$ and $\Delta_{b,c} = t^9(t-1)^9(t^2-t+1)^3(t^3-6t^2 + 3t + 1)$. Substituting into equation (\ref{eq:3}) of Section \ref{section:3} we obtain
\begin{equation}\label{eq:10}
t(t-1)^5(t^3-6t^2 + 3t + 1)\overline{a}_3^4 = 2^{12} d^6\Delta_1(t^2-t+1).
\end{equation}
Note that in this hypothesis $3|d_K$ implies that $E$ has additive reduction (see Proposition \ref{prop:5}). So we can assume that $E$ has additive reduction at $3$ and it is enough to prove that 
$$
3\nmid d_K \iff \text{the Kodaira symbol of $E$ at $3$ is $IV$}.
$$
\\~\\
\noindent\fbox{$\Longrightarrow$}
Assume that $3\nmid d_K$. We divide the proof of this implication according to the valuation of $t$.
\\~\\
\noindent \fbox{$v_3(t)<0$}
Taking the $3$-adic valuation of equation \eqref{eq:10}, we obtain
\begin{equation}\label{eq:11}
    4v_3(y_1) = v_3(\Delta_1) - 7v_3(t) \geq v_3(f(E/\mathbb{Q}_p))-7v_3(t)\geq 9,
\end{equation}
where the inequalities are deduced from Ogg's formula (see Theorem \ref{thm:6}) and the additive redction of $E$. Hence $v_3(y_1)\geq 3$. Taking valuations in the relations defining $b$ and $c$ yields the equations
\begin{equation}\label{eq:12}
    5v_3(t) = 3v_3(\overline{a}_2)-2v_3(y_1)
\end{equation}
\begin{equation}\label{eq:13}
    3v_3(t) = v_3(s)+v_3(\overline{a}_2)-v_3(y_1)
\end{equation}
Multiplying equation \eqref{eq:13} by $2$ and subtracting equation \eqref{eq:12}, we obtain
\begin{equation}\label{eq:14}
    v_3(t) = 2v_3(s) - v_3(\overline{a}_2).    
\end{equation}
If $v_3(s)<0$, then $v_3(\overline{a}_2) = 2v_3(s)$, which implies $v_3(t) = 0$, a contradiction. Therefore, $v_3(s)\geq 0$ and $v_3(\overline{a}_2) > 0$ (since $v_3(t)<0$). Since $v_3(\overline{a}_3) = v_3(y_1)>0$, the equation
$$
E_4: y^2+\overline{a}_1xy + \overline{a}_3y = x^3 + \overline{a}_2x^2
$$
has integral coefficients. Since $3$ is unramified in $K$, Remark \ref{rmk:3} implies that the equation of $E_4$ is minimal. Reducing modulo $3$, the coefficient $\overline{a}_3$ is $0$, and therefore the reduced curve has the singular point $(0,0)$. Tate's algorithm can be applied, and Step $2$ implies that
$v_3(\overline{a}_1^2+4\overline{a}_2)>0$.
Hence $v_3(s)>0$. Equation \eqref{eq:13} implies that $v_3(\overline{a}_3) = v_3(y_1) \geq 3$. Since $E_4$ is minimal, $v_3(\overline{a}_2), v_3(\overline{a}_1) = v_3(s)>0$, then  $v_3(\overline{a}_2) = 1$. Substituting in equation \eqref{eq:14}, we get 
$$
v_3(t) = 2v_3(s) - 1> 0,
$$
a contradiction.
\\~\\
\noindent \fbox{$v_3(t)>0$} Taking the $3$-adic valuation in equation \eqref{eq:10}, we obtain 
\begin{equation}\label{eq:15}
    v_3(t) + 4v_3(y_1) = v_3(\Delta_1) .
\end{equation}
Taking the $3$-adic valuation of $b$ and $c$, we obtain
$$
0<2v_3(t) = v_3(c) = v_3(b) = 3v_3(\overline{a}_2) - 2v_3(y_1).
$$
This implies $v_3(\overline{a}_2)>0$, which in turn implies $v_3(s)\geq 0$. Since $0<v_3(c) = v_3\left(1-\dfrac{\overline{a}_1\overline{a}_2}{\overline{a}_3}\right)$, it follows that
$$
0<v_3(\overline{a}_1) + v_3(\overline{a}_2) = v_3(\overline{a}_3) = v_3(y_1). 
$$
Since $v_3(s)\geq 0$, we have
$$
v_3(8d^2b_4 + 2db_2x_1 + 3x_1^2) \geq v_3(y_1)>0.
$$
Therefore the point $(x_1,y_1)$ reduces to a singular point. Note that $E_4$ has integral coefficients and is minimal by Remark \ref{rmk:3}.  Hence we can apply Tate's algorithm to the curve $\mathcal{T}_{b,c}$. Step $2$ yields the contradiction
$$
0 = v_3((1-c)^2-4b)>0.
$$
\\~\\
\noindent \fbox{$v_3(t)=0$}
Taking the $3$-adic valuation in equation \eqref{eq:10}, we obtain
\begin{equation}\label{eq:16}
    5v_3(t-1) +v_3(t^3-6t^2+3t+1) + 4v_3(y_1) = v_3(\Delta_1) + v_3(t^2-t+1) .
\end{equation}
If $v_3(t-1)>0$, then $v_3(b), v_3(c)>0$. Therefore one can apply exactly the same reasoning as in the case $v_3(t)>0$ and obtain a contradiction. Hence equation \eqref{eq:16} becomes
\begin{equation}\label{eq:17}
    v_3(t^3-6t^2+3t+1) + 4v_3(y_1) = v_3(\Delta_1) + v_3(t^2-t+1).
\end{equation}
Since $v_3(t)= v_3(t-1) = 0$, we have $t = -1 + 3k$ with $k\in\mathbb{Q}$ a $3$-adic integer. Substituting into $t^3-6t^2+3t+1$ and $t^2-t+1$, one easily deduces that
\begin{eqnarray*}
    v_3(t^3-6t^2+3t+1) &=& 2, \\
    v_3(t^2-t+1) &=& 1.
\end{eqnarray*}
Hence equation \eqref{eq:17} gives
\begin{equation*}
    1 + 4v_3(y_1) = v_3(\Delta_1)\geq 2,
\end{equation*}
where the inequality is deduced from Ogg's formula again (see Theorem \ref{thm:6}). From this we deduce that $v_3(b)=1$, $v_3(c) = 0$, and $v_3(y_1)>0$. Since
$$
0<v_3(b) = 3v_3(\overline{a}_2)-2v_3(y_1),
$$
it follows that $v_3(\overline{a}_2)>0$, which implies $v_3(s)\geq 0$, and hence
$$
v_3(8d^2b_4 + 2db_2x_1 + 3x_1^2) \geq v_3(y_1)>0.
$$
Therefore the point $(x_1,y_1)$ reduces to a singular point on $E_3$, which is a minimal equation by Remark \ref{rmk:3}. Hence we can apply Tate's algorithm to the curve with equation $\mathcal{T}_{b,c}$. Running the algorithm reaches Step $5$, where since the coefficient $b_6$ of $\mathcal{T}_{b,c}$ is $b^2$ and $v_3(b) = 2$, it follows that $\mathcal{T}_{b,c}$ has Kodaira symbol $IV$ at $3$. Since the original curve $E$ is a quadratic twist of this curve unramified at $3$, its Kodaira symbol is the same.
\\~\\
\noindent \fbox{$\Longleftarrow$}
Assume that the Kodaira symbol of $E$ at $3$ is $IV$. By contradiction, suppose that $3$ ramifies in $K$, so that $3\mid d$. We start bounding the valuation of $v_3(\Delta_1)$. Theorem \ref{thm:6} gives 
$$
v_3(\Delta_1) = f_3(E/\mathbb{Q})+ m_3(E/\mathbb{Q}) -1,
$$
Since the Kodaira symbol of $E$ is $IV$, we have $m_3(E/\mathbb{Q}) = 3$ (see \cite[IV.Table 4.1]{Silverman1994AdvancedTI}). On the other hand, $f_3(E/\mathbb{Q})\leq 5$ (see \cite[IV.Theorem 10.4]{Silverman1994AdvancedTI}). Therefore, $v_3(\Delta_1) \leq 7$. Now, we divide the proof according to the sign of the valuation of $t$.
\\~\\
\noindent \fbox{$v_3(t)<0$} Taking the $3$-adic valuation in equation \eqref{eq:4}, we obtain 
\begin{equation}\label{eq:18}
    4v_3(y_1) = 6 +  v_3(\Delta_1) - 7v_3(t)\geq 15 .
\end{equation}
Hence $v_3(y_1)\geq 4$. Moreover, since $v_3(t)<0$, the relations in equations \eqref{eq:12} and \eqref{eq:13} hold.  As in the case $v_3(t)<0$ in the previous implication, we obtain $v_3(s)\geq 0$. In fact, $v_3(s)>0$. Indeed, if $v_3(s)=0$, then
$$
v_3(\overline{a}_2) = v_3(-s^2+3x_1+db_2) = 0.
$$
Thus equations \eqref{eq:12} and \eqref{eq:13} become 
\begin{eqnarray*}
    5v_3(t) &=& -2v_3(y_1)\\
    3v_3(t) &=& -v_3(y_1)
\end{eqnarray*}
Hence $v_3(t)= 0$ or $v_3(y_1)=0$, a contradiction.
\\~\\
Multiplying equation \eqref{eq:7} by $2$ and subtracting equation \eqref{eq:6}, we obtain 
$$
0 < -v_3(t) = v_3(\overline{a}_2)-2v_3(s),$$
and therefore $v_3(\overline{a}_2)\geq 3$. Thus, we obtain $v_3(\overline{a}_i)\geq i$. Consequently, the equation of $E_4$ is not minimal. The valuation of the discriminant of the equation of $E_4$ is $v_3(\Delta_1)+6$. Since $v_3(\Delta_1)\leq 7$ and we can perform a change of variables in the equation of $E_4$ that lowers the valuation of the discriminant to $v_3(\Delta_1) -6$ while keeping the coefficients integral, we obtain $v_3(\Delta_1)\in\{6,7\}$. Hence, the equation
$$
y^2+\frac{\overline{a}_1}{3}xy + \frac{\overline{a}_3}{3^3} = x^3+\frac{\overline{a}_2}{3^2}x^2
$$
has $3$-integral coefficients, and its discriminant has valuation $0$ or $1$. We compute
\begin{eqnarray*}
    \overline{b}_2 &=& \frac{\overline{a}_1^2}{3^2} + 4\frac{\overline{a}_2}{3^2}\\
    \overline{b}_4 &=& \frac{\overline{a}_1\overline{a}_3}{3^4}\\
    \overline{b}_6 &=& \frac{\overline{a}_3^2}{3^6}\\
    \overline{b}_8 &=& \frac{\overline{a}_2\overline{a}_3^2}{3^8}
\end{eqnarray*}
and we have $v_3(\overline{b}_4)\geq 1, v_3(\overline{b}_6)\geq 2 $ and $v_3(\overline{b}_8)\geq 3$, because we have already seen that $v_3(\overline{a}_1)\geq 1$, $v_3(\overline{a}_2)\geq 3$ and $v_3(\overline{a}_3)\geq 4$. Hence the valuation of its discriminant is
$$
v_3(-\overline{b}_2^2\overline{b}_8-8\overline{b}_4^3-27\overline{b}_6^2 + 9\overline{b}_2\overline{b}_4\overline{b}_6)\geq 2.
$$
This is a contradiction.
\\~\\
\noindent \fbox{$v_3(t)>0$} Arguing as at the beginning of this case in the previous implication, we obtain
\begin{equation}\label{eq:19}
    0<2v_3(t)=v_3(b) = 3v_3(\overline{a}_2) - 2v_3(y_1),
\end{equation}
\begin{equation}\label{eq:20}
    0<v_3(\overline{a}_1) + v_3(\overline{a}_2) = v_3(y_1),
\end{equation}
together with $v_3(\overline{a}_2),v_3(y_1)>0$ and $v_3(s)\geq 0$. In fact $v_3(s)>0$ since $0<v_3(\overline{a}_2) = v_3(-s^2+3x_1+db_2)$, noting that $v_3(d) = 1$. Since $v_3(\overline{a}_2)$ and $v_3(s)>0$, equation \eqref{eq:20} implies that $v_3(y_1)\geq 2$. In turn, equation \eqref{eq:19} allows us to deduce that $v_3(\overline{a}_2)\geq 2$. Looking again at equation \eqref{eq:20}, we obtain $v_3(y_1)\geq 3$. Iterating this argument, we obtain $v_3(\overline{a}_2)\geq 3$ and $v_3(y_1)\geq 4$.
\\~\\
In particular, the equation of $E_4$ is not minimal. From this point one can repeat exactly the same reasoning as in the previous case and obtain a contradiction.
\\~\\
\noindent \fbox{$v_3(t)=0$} Since we are considering $3$-adic valuations, we have either $v_3(t-1) >0$ or $v_3(t+1)>0$. Taking the $3$-adic valuation of equation \eqref{eq:4}, we obtain
\begin{equation}\label{eq:21}
    5v_3(t-1)+v_3(t^3-6t^2+3t+1) + 4v_3(y_1) = v_3(\Delta_1) + v_3(t^2-t+1)+6.
\end{equation}
We begin by seeking a contradiction in the case $v_3(t+1)>0$. In this case we already saw in the previous implication that
\begin{eqnarray*}
    v_3(t^3-6t^2+3t+1) &=& 2\\
     v_3(t^2-t+1) &=& 1
\end{eqnarray*}
Thus, equation \eqref{eq:21} becomes 
$$
4v_3(y_1) = 5 + v_3(\Delta_1)\leq 5+ 7 =12,
$$
where the inequality is deduced from Ogg's formula (see Themorem \ref{thm:6}). Hence, $v_3(y_1)\in\{2,3\}$. Taking the valuation of $b$, we obtain 
$$
1 = 3v_3(\overline{a}_2)-2v_3(y_1).
$$
Neither of the two possible values of $v_3(y_1)$ satisfies this equation. Hence we obtain a contradiction.
\\~\\
We now seek a contradiction assuming $v_3(t-1)>0$.
Equation \eqref{eq:21} simplifies to
$$
5v_3(t-1)+ 4v_3(y_1) = 6 + v_3(\Delta_1)\leq 13
$$
We claim that $v_3(t-1)= 1$. Because of the previous inequality, it is enough to show that $v_3(y_1)>0$. Taking the valuation of $b$,
$$
0<v_3(t-1) = 3v_3(\overline{a}_2)-2v_3(y_1).
$$
we get $v_3(\overline{a}_2)>0$. Since $0<v_3(c) = v_3\left(1-\dfrac{\overline{a}_1\overline{a}_2}{\overline{a}_3}\right)$, we have $v_3(y_1) = v_3(\overline{a}_1) + v_3(\overline{a}_2)> 0$ as desired. Then equation \eqref{eq:21} simplifies to 
$$
0<4v_3(y_1) =1 + v_3(\Delta_1)\leq 8
$$
Hence $v_3(y_1)\in\{1,2\}$. Taking the valuation of $b$ again, we obtain
$$
1 = 3v_3(\overline{a}_2)-2v_3(y_1).
$$
It is easy to check that given the restriction on $v_3(y_1)$, the only possible values are $v_3(\overline{a}_2)= v_3(y_1) = 1$. Thus
$$
2 = 2v_3(y_1) = v_3(x_1^3+db_2x_1^2 + 8d^2b_4x_1 + 16d^3b_6).
$$
Since $v_3(d)>0$, we have $v_3(x_1)>0$. However, this shows that each summand inside the valuation in the previous equation has a valuation of at least $3$, which gives a contradiction.
\end{proof}

This concludes the proofs for the main theorems when $\ell = 3$. 
\begin{remark}\label{rmk:5}
    \begin{itemize}
        \item In the hypothesis of Theorem \ref{thm:5}, we can determine all the primes that ramify over $K$. Indeed, if $p\neq 3$ we can apply Theorem \ref{thm:4} with $\ell=3$. Note that if a point $P$ of order $9$ appears over $K$, the point $[3]P$ of order $3$ cannot be defined over $\mathbb{Q}$. This is a consequence of Theorem \ref{thm:1} 
        \item If we know that over $K$ there appears a point of order $3$ and we are in the cases where $E(K)[3] = \mathcal{C}_3\times \mathcal{C}_3$ or $E(K) = \mathcal{C}_9$, we can determine all the primes that ramify over $K$. If $E(K)[3]$ has a different form, we can use Theorem \ref{thm:4} to decide if the primes $p\neq 3$ ramify over $K$. 
        \\~\\
        We shall expect to not be able to determine wether $p=3$ ramifies over $K$ or not, because in the case of adding a point of order $3$ the field $K$ is not unique (see \cite[Theorem 2]{GonTorII}). The same result shows that at most there are two possibilities for $K$. If there are two different quadratic fields $K_1 = \mathbb{Q}(\sqrt{d_1})$ and $K_2 = \mathbb{Q}(\sqrt{d_2})$ such that $E(K_1)[3] \neq E(\mathbb{Q})[3] \neq E(K_2)[3]$, then $\mathbb{Q}(E[3]) = \mathbb{Q}(\sqrt{d_1}, \sqrt{d_2})$. The Galois diagram of subextensions is 
        
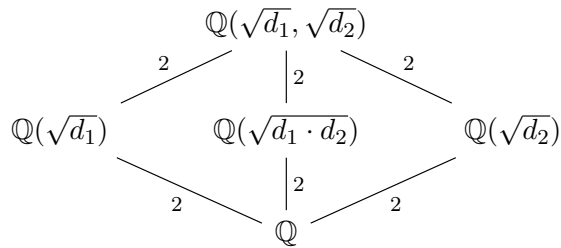
\[\begin{tikzcd}
	& {\mathbb{Q}(\sqrt{d_1}, \sqrt{d_2})} \\
	{\mathbb{Q}(\sqrt{d_1})} & {\mathbb{Q}(\sqrt{d_1\cdot d_2})} & {\mathbb{Q}(\sqrt{d_2})} \\
	& {\mathbb{Q}}
	\arrow["2", no head, from=2-1, to=1-2]
	\arrow["2"', no head, from=2-1, to=3-2]
	\arrow["2"', no head, from=2-2, to=1-2]
	\arrow["2", no head, from=2-2, to=3-2]
	\arrow["2"', no head, from=2-3, to=1-2]
	\arrow["2", no head, from=2-3, to=3-2]
\end{tikzcd}\]
Using the results of this section, we can determine if the primes greater or equal than $5$ divide $d_1$ and $d_2$. If there is $p\geq 5$ with that property, then $\mathbb{Q}(\sqrt{d_i})\neq \mathbb{Q}(\sqrt{-3})$. However, $\mathbb{Q}(\sqrt{-3})\subset\mathbb{Q}(E[3])$. Then, $\mathbb{Q}(\sqrt{d_1\cdot d_2}) = \mathbb{Q}(\sqrt{-3})$. So $3$ divides one of the $d_i$ and not the other one.
    \end{itemize}
\end{remark}

\section{Common results for the primes $\ell = 5,7$}
The strategy to prove theorem \ref{thm:4} for $\ell = 5,7$ consists of proceeding by contradiction and using the results presented in section \ref{section:2} about the image of inertia discard most of the cases. The remaining cases will be proven using similar techniques to the ones used in the previous section.
\\~\\
Let $E/\mathbb{Q}$ be an elliptic curve such that there exist a quadratic extension $K/\mathbb{Q}$ and a point $P\in E(K)[\ell]\setminus E(\mathbb{Q})$ with $\ell\in \{5,7\}$. Let $p$ be a prime such that $E$ has additive reduction at $p$. Suppose that $p$ is unramified in $K$. Let $\mathfrak{P}|p$ be a prime of $\mathbb{Q}(E[\ell])$ and set $L = \mathbb{Q}(E[\ell])^{I(\mathfrak{P}|p)}$. We then have the tower
$$
\begin{tikzcd}
{\mathbb{Q}(E[\ell])}                                          \\
L \arrow[u, "I(\mathfrak{P}|p)", no head] \arrow[d, no head] \\
K \arrow[d, "2", no head]                                    \\
\mathbb{Q}                                                  
\end{tikzcd}
$$

Sutherland \cite{Sutherland} and Zywina \cite{Zywina} have characterized all possible subgroups that can occur as $\overline{\rho}_{E, \ell}(G_{\mathbb{Q}})$ in the cases $\ell=5$ and $\ell=7$. We analyze each of these cases, ruling out in each one the possibility of having an $\ell$-torsion point defined over a quadratic extension $K/\mathbb{Q}$.
\\~\\
In this case, \cite{Zywina} shows that $\overline{\rho}_{E,\ell}(G_{\mathbb{Q}})$ is conjugate in $\GL_2(\mathbb{F}_\ell)$ to a group from a list of 15 possible groups (see \cite[Theorem 1.4]{Zywina}). Moreover, Sutherland studies these groups and determines in each case the index of the largest subgroup fixing a non-zero vector in $\mathbb{F}_\ell^2$; this number coincides with the degree of the minimal extension $K/\mathbb{Q}$ over which $E$ has a rational $\ell$-torsion point. We use this to prove the following result.
\begin{proposition}\label{prop:6}
     Let $E/\mathbb{Q}$ be a rational elliptic curve such that there exists a point $P\in E(K)[5]\setminus E(\mathbb{Q})$, where $K = \mathbb{Q}(\sqrt{d})$. Let $p\neq 5$ be a prime number such that $E$ has additive reduction at $p$. If $p$ does not ramify in $K$, then $K=\mathbb{Q}(\sqrt{5})$. 
\end{proposition}
\begin{proof}
First, note that  $\mathbb{Q}(\zeta_5)\subset L$, because $p\neq 5$ and the only prime that ramifies in $\mathbb{Q}(\zeta_5)$ is $5$. Then, $[L:K]\geq 2$.
\\~\\
 In Table 3 (page 64 of \cite{Sutherland}), we can find the list of these 15 groups together with their indices. Only four of them have this degree equal to $2$, namely those labeled 5Cs.1.3, 5Cs.4.1, 5B.1.4 and 5B.4.1. We examine each of these cases:

\begin{itemize}

\item 5Cs.1.3 is a cyclic group of order $4$. Since the reduction is additive at $p$, the image of inertia cannot be trivial. On the other hand, $\mathbb{Q}(\zeta_5)\subset \mathbb{Q}(E[5])$. As $[\mathbb{Q}(E[5]):\mathbb{Q}] = 4$, it follows that $\mathbb{Q}(\zeta_5) = \mathbb{Q}(E[5])$. Since $p\neq 5$, the extension $\mathbb{Q}(\zeta_5)/\mathbb{Q}$ is unramified at $p$, whereas $\mathbb{Q}(E[5])/\mathbb{Q}$ is ramified because the inertia at $p$ is non-trivial, which is a contradiction.

\item 5Cs.4.1. It is a group of order $8$, generated by the matrices
$$
\begin{pmatrix} 4 & 0\\ 0 & 4\end{pmatrix}, 
\begin{pmatrix} 1 & 0\\ 0 & 2\end{pmatrix}.
$$

Since the group has order $8$, we have the diagram
$$
\begin{tikzcd}
{\mathbb{Q}(E[5])}                                          \\
L \arrow[u, "I(\mathfrak{P}|p)", no head] \arrow[d,"2", no head] \\
K \arrow[d, "2", no head]                                    \\
\mathbb{Q}                                                  
\end{tikzcd}
$$

Since $\mathbb{Q}(\zeta_5)\subset L$ and since the inertia is non-trivial, we must have $\mathbb{Q}(\zeta_5)= L$. Hence $K = \mathbb{Q}(\sqrt{5})$.

\item 5B.1.4. It is a subgroup of order $20$, generated by the matrices
\begin{equation*}
\left\{ 
\begin{pmatrix} 4 & 0\\ 0 & 3\end{pmatrix}, 
\begin{pmatrix} 1 & 1\\ 0 & 1\end{pmatrix} 
\right\}.
\end{equation*}

We have the diagram
$$
\begin{tikzcd}
{\mathbb{Q}(E[5])}                                          \\
L \arrow[u, "I(\mathfrak{P}|p)", no head] \arrow[d, no head] \\
K \arrow[d, "2", no head]                                    \\
\mathbb{Q}                                                  
\end{tikzcd}
$$

We have
$$
4\cdot 5 = [\mathbb{Q}(E[5]):\mathbb{Q}]
= [\mathbb{Q}(E[5]):L]\,[L:\mathbb{Q}(\zeta_5)]\cdot 4.
$$

Hence
$$
5 = [\mathbb{Q}(E[5]):L]\,[L:\mathbb{Q}(\zeta_5)]
= \#I(\mathfrak{P}|p)\,[L:\mathbb{Q}(\zeta_5)].
$$

This implies that the inertia must have order $5$, since we know it is non-trivial. On the other hand, since $E$ has additive reduction, $\Phi_p \neq \{1\}$ and therefore $\#I(\mathfrak{P}|p)$, which is divisible by $\#\Phi_p$, is a multiple of $2,3,4$ or $6$ (see Proposition \ref{prop:3}). Thus, inertia cannot have order $5$, which is a contradiction.

\item 5B.4.1. It is a subgroup of order $40$, generated by
\begin{equation*}
\left\{ 
\begin{pmatrix} 4 & 0\\ 0 & 4\end{pmatrix}, 
\begin{pmatrix} 1 & 0\\ 0 & 2\end{pmatrix}, 
\begin{pmatrix} 1 & 1\\ 0 & 1\end{pmatrix} 
\right\}.
\end{equation*}
 Reasoning as in the previous case we obtain
$$
2\cdot 5
=
[\mathbb{Q}(E[5]):L]\,[L:\mathbb{Q}(\zeta_5)]
=
\#I(\mathfrak{P}|p)\,[L:\mathbb{Q}(\zeta_5)].
$$
Again, the inertia cannot be trivial and because of Proposition \ref{prop:3}, $\#I(\mathfrak{P}|p)$ has to be a multiple of $2,3,4$ o $6$. That implies  $\#I(\mathfrak{P}|p) = 2$ or $\#I(\mathfrak{P}|p) = 10$. In both cases, it is easy to see that  $K\subset \mathbb{Q}(\zeta_5)$ and again
$K = \mathbb{Q}(\sqrt{5})$.
\end{itemize}
\end{proof}

We can do a similar analysis for the prime $\ell = 7$, obtaining the following result.
\begin{proposition}\label{prop:7}
     Let $E/\mathbb{Q}$ be a rational elliptic curve such that there exists a point $P\in E(K)[7]\setminus E(\mathbb{Q})$, where $K = \mathbb{Q}(\sqrt{d})$. Let $p\neq 7$ be a prime number such that $E$ has additive reduction $p$. If $p$ does not ramify in $K$, then $K=\mathbb{Q}(\sqrt{-7})$. 
\end{proposition}
\begin{proof}
Again, \cite{Zywina} studies this case and shows that $\overline{\rho}_{E,7}(G_{\mathbb{Q}})$ is conjugate in $\GL_2(\mathbb{F}_7)$ to a group from a list of 16 possible groups (see \cite[Theorem 1.5]{Zywina}). Sutherland studies these groups and determines in each case the index of the largest subgroup that fixes a nonzero vector in $\mathbb{F}_{7}^2$; this quantity coincides with the degree of the minimal extension $K/\mathbb{Q}$ over which $E$ has a rational $7$-torsion point.
\\~\\
In Table 3 (page 65 of \cite{Sutherland}), we can find the list of these 16 groups together with these indices. There are only two of them for which this degree is $2$, namely those labeled 7B.1.6 and 7B.6.1. Let us examine each of these cases:
\begin{itemize}
    \item 7B.1.6 is a group of order 42, generated by the matrices
    \begin{equation*}
        \begin{pmatrix} 6 & 0\\ 0 & 4\end{pmatrix}, \begin{pmatrix} 1 & 1\\ 0 & 1\end{pmatrix}
    \end{equation*}
    Similarly to the case $l = 5$, we have that $\mathbb{Q}(\zeta_7)\subset L$. Since the group has order $42$, we have
    $$
    7\cdot6 = [\mathbb{Q}(E[7]):L][L:\mathbb{Q}] = \#I(\mathfrak{P}|p)\cdot[L:\mathbb{Q}(\zeta_7)]\cdot[\mathbb{Q}(\zeta_7):\mathbb{Q}] = \#I(\mathfrak{P}|p)\cdot[L:\mathbb{Q}(\zeta_7)]\cdot 6.
    $$
    Inertia cannot be trivial because the reduction of $E$ is additive at $p$. Then, $\#I(\mathfrak{P}|p) = 7$. However, $\Phi_p\subset I(\mathfrak{P}|p)$ is not trivial and Proposition \ref{prop:3} implies that $7\nmid \#\Phi_p$, which is a contradiction.

    \item 7B.6.1 is a group of order 84, generated by the three matrices
    \begin{equation*}
        \begin{pmatrix} 6 & 0\\ 0 & 6\end{pmatrix},  
        \begin{pmatrix} 1 & 0\\ 0 & 3\end{pmatrix},
        \begin{pmatrix} 1 & 1\\ 0 & 1\end{pmatrix}
    \end{equation*}
    Again, $\mathbb{Q}(\zeta_7)\subset L$, so we have the relation
    $$
    7\cdot 2\cdot6 = \#I(\mathfrak{P}|p)\cdot[L:\mathbb{Q}(\zeta_7)]\cdot [\mathbb{Q}(\zeta_7):\mathbb{Q}] = \#I(\mathfrak{P}|p)\cdot[L:\mathbb{Q}(\zeta_7)]\cdot 6.
    $$
    
    Hence, $\#I(\mathfrak{P}|p) \mid 2\cdot 7$. Proposition \ref{prop:3} shows that $7\nmid \#I(\mathfrak{P}|p)$. Therefore, we obtain that $\#I(\mathfrak{P}|p) = 2$ or $ \#I(\mathfrak{P}|p) = 14$. In this way, $L$ contains only one quadratic extension, which must be the same one contained in $\mathbb{Q}(\zeta_7)$, namely $\mathbb{Q}(\sqrt{-7}) = K$.
\end{itemize}    
\end{proof}

So we have proven the following result
\begin{proposition}\label{prop:8}
     Let $E/\mathbb{Q}$ be a rational elliptic curve such that there exists a point $P\in E(K)[7]\setminus E(\mathbb{Q})$, where $K = \mathbb{Q}(\sqrt{d})$. Let $p\neq 7$ be a prime number such that $E$ has additive reduction $p$. If $p$ does not ramify in $K$, then $K=\mathbb{Q}(\sqrt{-7})$. 
\end{proposition}
Proposition \ref{prop:8} and Proposition \ref{prop:6} are the conclusions of the analysis of the image of inertia. In order to go from them to the conclusions of Theorem \ref{thm:4}, we need to use the techniques used in the proof of Theorem \ref{thm:5}. 
\\~\\
We separate the proofs prime by prime
\section{Prime $\ell = 5$}
We start by refining Proposition \ref{prop:6}. Our goal is to prove that $K$ cannot be $\mathbb{Q}(\sqrt{5})$ and conclude that the statement of Theorem \ref{thm:4} for $p\neq 5$ is true. We begin with two lemmas for the case $p=2$.
\begin{lemma}\label{lemma:3}
Let $E/\mathbb{Q}$ be an elliptic curve with a minimal model at $2$ of the form
$$
E: y^2 =  x^3 +a_2x^2 + a_4x + a_6.
$$
 Suppose that there exists $P\in E(\mathbb{Q}(\sqrt{5}))[5]\setminus E(\mathbb{Q})$, then $E$ does not have additive reduction at $2$.
\end{lemma}

\begin{proof}
    By contradiction suppose that $E$ has additive reduction at $2$. By hypothesis, we have a model
    $$
    E_1: y^2 = x^3 +a_2x^2 + a_4x + a_6.
    $$
     Applying the changes of variables described in Remark \ref{rmk:2}, we obtain the curve
$$
\mathcal{T}_{b,c}:  y^2+ (1-c)xy - by = x^3 - bx^2.
$$
Moreover, since the point has $5$-torsion, $b=c$ and
$$
\Delta_{b,c} = b^5(b^2-11b-1).
$$
Substituting into equation \eqref{eq:3} and taking the $2$-adic valuation we obtain
\begin{equation}\label{eq:22}
    v_2(b) + v_2(b^2- 11b -1) + 4 + 4v_2(y_1) = v_2(\Delta_1).
\end{equation}
We split the argument into cases depending on the valuation of $b$.
\\~\\
\noindent \fbox{$v_2(b)>0$}
\\~\\
In this case, the curve $\mathcal{T}_{b,c}$ has $2$-integral coefficients (integers in $\mathbb{Q}_2$). We apply Tate's algorithm. To apply the first steps, we need the coefficients of the curve to be $2$-integral and the singular point mentioned in Step 2 to be $(0,0)$. 
\\~\\
In our case, the point $(0,0)$ (which is the image of the point $Q$ on the curve $E_3$, see Remark \ref{rmk:2}) is a point of $5$-torsion. Remark \ref{rmk:4} shows that it must reduce to a singular point.
\\~\\
Step $2$ of Tate's algorithm applied to our curve yields that since the reduction at $2$ is additive, then
$$
0<v_2((1-b)^2-4b) = 0,
$$
which is a contradiction.
\\~\\
\noindent \fbox{$v_2(b)=0$}
\\~\\
This case is impossible because, as we have seen before, the point $(0,0)$ is singular. Differentiating with respect to $y$ in the equation of $\mathcal{T}_{b,c}$ and evaluating at $(0,0)$, we deduce that $v_2(b)>0$, a contradiction.
\\~\\
\noindent \fbox{$v_2(b)<0$}
\\~\\
This implies that equation \eqref{eq:22} becomes
$$
3v_2(b) + 4 +4v_2(y_1) = v_2(\Delta).
$$
Since $v_2(b)< 0$, we have $v_2(y_1)\geq 1$, because otherwise $v_2(\Delta)\leq 1$, which is impossible since the reduction is additive. 
\\~\\
Since the point $(0,0)$ has order $5$, we have $b=c$. Together with $v_2(b)<0$, this gives
$$
3v_2(\overline{a}_2)- 2v_2(\overline{a}_3) = v_2(b) = v_2(c) = v_2\left(1- \dfrac{\overline{a}_1\overline{a}_2}{\overline{a}_3}\right) = v_2\left(\dfrac{\overline{a}_1\overline{a}_2}{\overline{a}_3}\right).
$$
Substituting the definitions of $\overline{a}_i$ for $i\in\{1,2,3\}$ we obtain
\begin{equation}\label{eq:23}
    2v_2(\overline{a}_2) = 2v_2(-s^2+3x_1 + 5a_2) = v_2(3x_1^2 + 10a_2x_1 + 5^2a_4) + 1.
\end{equation}
In particular, we obtain $v_2(\overline{a}_2)>0$, which implies that $v_2(s)\geq 0$. We have $v_2(\overline{a}_1),v_2(\overline{a}_2),v_2(\overline{a}_3)\geq 0$. This implies that the coefficients of $E_4$ are integers. Since $E_3$ was a minimal model and the change from $E_3$ to $E_4$ does not change the discriminant, it follows that $E_4$ is also a minimal model, as it has integral coefficients.
\\~\\
Now observe that if $v_2(\overline{a}_1) \geq 1$, $v_2(\overline{a}_2) \geq 2$ and $v_2(\overline{a}_3) \geq 3$, then we could make a change of variables leaving the coefficients integral and decreasing the valuation of the discriminant by a factor of $12$, but this is impossible because the equation of $E_4$ is minimal.
\\~\\
From the definition $\overline{a}_1 = 2s$, we know that $v_2(\overline{a}_1)\geq 1$, since we had already seen that $v_2(s)\geq 0$. On the other hand, we know that
\begin{eqnarray*}
v_2(\overline{a}_2)\geq 1 \\
v_2(\overline{a}_3) = 1 + v_2(y_1)\geq 2
\end{eqnarray*}
Suppose that $v_2(\overline{a}_2)\geq 2$. Since $E_4$ is a minimal model, $v_2(y_1) = 1$. But
$$
0>v_2(b) = 3v_2(\overline{a}_2) -2(v_2(y_1)+1)= 3v_2(\overline{a}_2) -4 \geq 2,
$$
a contradiction. Therefore $v_2(\overline{a}_2) = 1$. Using equation \eqref{eq:23} we obtain
$$
v_2(3x^2 + 10a_2x_1 + 25a_4) = 1.
$$
Now,
$$
0\leq v_2(s) = v_2(3x^2 + 10a_2x_1 + 25a_4)  -1 -v_2(y_1) = -v_2(y_1) <0,
$$
a contradiction.
\end{proof}
The following lemma allows us to discard $\mathbb{Q}(\sqrt{5})$ altogether, when $p = 2$.

\begin{lemma}\label{lemma:4}
Let $E/\mathbb{Q}$ be an elliptic curve. Suppose that there exists $P\in E(\mathbb{Q}(\sqrt{5}))[5]\setminus E(\mathbb{Q})$, then $E$ does not have additive reduction at $2$.
\end{lemma}

\begin{proof}
    By contradiction, suppose that $E$ has additive reduction at $2$. We can always find a minimal equation for $E$
    $$
    E_1 : y^2 + a_1xy + a_3y = x^3 + a_2x^2 + a_4x + a_6
    $$
    with integer coefficients and such that $a_1,a_3\in\{0,1\}$ and $a_2\in\{-1,0,1\}$. After performing the procedure explained in Section \ref{section:3}, we obtain the twist of $E_1$ in Tate normal form
$$
\mathcal{T}_{b,c}:  y^2+ (1-c)xy - by = x^3 - bx^2
$$
with $b :=- \dfrac{\overline{a}_2^3}{\overline{a}_3^2}$ and $c:= 1- \dfrac{\overline{a}_1\overline{a}_2}{\overline{a}_3}$. Moreover, since the point has $5$-torsion, $b=c$ and
$$
\Delta_{b,c} = b^5(b^2-11b-1),
$$
(see Section \ref{section:3}). Substituting into equation \eqref{eq:2}  and applying the $2$-adic valuation we obtain
\begin{equation*}
    v_2(b) + v_2(b^2-11b-1)+ 4v_2(\overline{a}_3) = 12 + v_2(\Delta_1).
\end{equation*}
Using the definition of $\overline{a}_3$, we obtain
\begin{equation}\label{eq:24}
        v_2(b) + v_2(b^2-11b-1)+ 4v_2(y_1) = 8 + v_2(\Delta_1).
\end{equation}
Since the curve $E_3$ has additive reduction at $2$ and the point $(x_1,y_1)$ has order $5$, this point must reduce to a singular point (see Remark  \ref{rmk:4}). Taking the partial derivatives in the equation of $E_3$ and setting them equal to $0$ modulo $2$, we obtain that $v_2(3x_1^2)>0$, hence $x_1$ is even. The equation of $E_3$ directly implies that $y_1$ is also even. 
\\~\\
Moreover, since the reduction of the original curve is additive, applying the formulas of \cite[III.1]{Silverman} and Proposition \cite[VII.5.1]{Silverman}, we have $0<v_2(c_4) = v_2(b_2^2-24b_4)$. The definition of $b_2$ is
$$
b_2 = a_1^2 +4a_2.
$$
Since $a_1\in\{0,1\}$, it follows that $a_1$ must be $0$, because otherwise $b_2$ would be odd and $c_4$ would also be odd. Applying now the equation of $E_3$ at the point $(x_1,y_1)$, we obtain
$$
2v_2(y_1) = v_2(x_1^3 + 5b_2x_1^2 +8\cdot5^2b_4x_1 + 16\cdot 5^3b_6)\geq 3.
$$
Hence $v_2(y_1)\geq 2$ and therefore
$$
v_2(x_1^3 + 5b_2x_1^2 +8\cdot5^2b_4x_1 + 16\cdot 5^3b_6)\geq 4.
$$
Using that $v_2(b_2) \geq 2$, we obtain that $v_2(x_1)\geq 2$. We divide the proof into cases as in the proof of the previous theorem, that is according to the sign of the valuation of $b$.
\\~\\
\noindent \fbox{$v_2(b)>0$}
\\~\\
As in the previous lemma, we can apply Tate's algorithm to the curve
$$
\mathcal{T}_{b,c}:  y^2+ (1-c)xy - by = x^3 - bx^2.
$$
Step $2$ of Tate's algorithm applied to our curve yields that since the reduction at $2$ is additive, then
$$
0<v_2((1-b)^2-4b) = 0,
$$
which is a contradiction.
\\~\\
\noindent \fbox{$v_2(b)=0$}
\\~\\
This case is impossible because, since the point $(0,0)$ is singular, differentiating with respect to $y$ in the equation of $\mathcal{T}_{b,c}$ and evaluating at $(0,0)$, we deduce that $v_2(b)>0$, a contradiction.
\\~\\
\noindent \fbox{$v_2(b)<0$}
\\~\\
In this case, using that $b=c$, we obtain
$$
3v_2(\overline{a}_2)- 2v_2(\overline{a}_3) = v_2(b) = v_2(c) = v_2\left(1- \dfrac{\overline{a}_1\overline{a}_2}{\overline{a}_3}\right) = v_2\left(\dfrac{\overline{a}_1\overline{a}_2}{\overline{a}_3}\right).
$$
Substituting the definitions of $\overline{a}_i$ for $i\in\{1,2,3\}$ we obtain
\begin{equation}\label{eq:25}
    2v_2(\overline{a}_2) = 2v_2(-s^2+3x_1 + 5a_2) = v_2(3x_1^2 + 10a_2x_1 + 25a_4) + 1.
\end{equation}
We now study the valuation of $\overline{a}_2$. Substituting the definition of $s$ and taking the $2$-adic valuation we obtain
\begin{eqnarray*}
    v_2(\overline{a}_2) &=& v_2(100b_2^2x_1^2 + 60b_2x_1^3 + 9x_1^4 + 4000b_2b_4x_1 + 1200b_4x_1^2 - 4db_2y_1^2 - 12x_1y_1^2 + 40000b_4^2) \\
    &-& 2(v_2(y_1)+1).
\end{eqnarray*} 
Looking at the valuation of each term in the first summand on the right-hand side, one checks that
$$
v_2(100b_2^2x_1^2 + 60b_2x_1^3 + 9x_1^4 + 4000b_2b_4x_1 + 1200b_4x_1^2 - 4db_2y_1^2 - 12x_1y_1^2 + 40000b_4^2)\geq 8.
$$
Applying the definition of $b$, we obtain
\begin{eqnarray*}
    0>v_2(b) &=& 3v_2(4b_2^2x_1^2 + 12b_2x_1^3 + 9x_1^4 + 32b_2b_4x_1 + 48b_4x_1^2 - 20b_2y_1^2 - 12x_1y_1^2 + 64b_4^2)\\
&-& 8(v_2(y_1)+1).
\end{eqnarray*}
In particular, $v_2(y_1)\geq 3$, otherwise $v_2(b)\geq 0$, a contradiction. Returning to the relation between $x_1$ and $y_1$,
$$
v_2(x_1^3 +5b_2x_1^2 + 8\cdot 5^2 b_4x_1 + 16\cdot 5^3b_6)\geq 6.
$$
Since we already knew that $v_2(x_1)\geq 2$ and $v_2(b_2) = 2$, it follows that $v_2(b_6) \geq 1$. But $b_6 = a_3^2 + 4a_6$ with $a_3\in\{0,1\}$. Hence $a_3 = 0$ and since we had already seen that $a_1 = 0$, the previous lemma yields the desired contradiction.
\end{proof}
\begin{proposition}\label{prop:9}
Let $E/\mathbb{Q}$ be an elliptic curve and let $p\neq 5$ be a prime. Suppose that there exists $P\in E(\mathbb{Q}(\sqrt{5}))[5]\setminus E(\mathbb{Q})$. Then $E$ does not have additive reduction at $p$. 
\end{proposition}
\begin{proof}
First, by Lemma \ref{lemma:4}, we may assume that $p\neq 2$. We begin with the curve $E$ with a minimal model
$$
E_1 : y^2 + a_1xy + a_3y = x^3 + a_2x^2 + a_4x + a_6
$$
with integer coefficients such that $a_1,a_3\in\{0,1\}$ and $a_2\in\{-1,0,1\}$. 
\\~\\
We argue by contradiction and assume that $E$ has additive reduction at $p$. Following the procedure in Section \ref{section:2}, we obtain the twist of $E_1$ over $\mathbb{Q}\left(\sqrt{5}\right)$ in Tate normal form
$$
\mathcal{T}_{b,c}:  y^2+ (1-c)xy - by = x^3 - bx^2.
$$
with $b :=- \dfrac{\overline{a}_2^3}{\overline{a}_3^2}$ and $c:= 1- \dfrac{\overline{a}_1\overline{a}_2}{\overline{a}_3}$.
\\~\\
Note that for any $p\neq 5$, the reduction type of $\mathcal{T}_{b,c}$ at $p$ is the same as that of $E$, i.e. it is additive reduction, since the twist was taken in an extension unramified at $p$. Moreover, throughout the procedure of Section \ref{section:3} it remains a minimal model at $p$, because, as $p$ does not ramify in $\mathbb{Q}(\sqrt{5})$, $E_1$ is a minimal model over $\mathbb{Q}(\sqrt{5})$ and the $p$-adic valuation of the discriminant of $E_1$ is the same as that of the discriminant of $E_3$.
\\~\\
On the other hand, since the point $(0,0)$ on the curve $\mathcal{T}_{b,c}$ has order $5$, we have $b=c$ and $\Delta_{b,c} = b(b^2-11b-1)$. Substituting into equation \eqref{eq:2}, we obtain
$$
b(b^2-11b-1)\overline{a}_3^4 = 2^{12}5^6\Delta_1.
$$
Taking the $p$-adic valuation and using that $p\neq 2,5$, we obtain
$$
v_p(b) + v_p(b^2-11b-1) + 4v_p(y_1) = v_p(\Delta_1).
$$
We divide the proof according to the sign of the valuation of $b$.
\\~\\
\noindent \fbox{$v_2(b)\geq0$} 
\\~\\
 The point $(0,0)$ on the curve $E_4$ is singular since the reduction is additive (see Remark \ref{rmk:4}). Since $v_p(b)\geq 0$, we can apply Tate's algorithm to the curve $\mathcal{T}_{b,c}$. The first two steps imply that
$$
v_p(b)>0 \quad \text{and} \quad v_p((1-b)^2-4b)>0,
$$
which is a contradiction.
\\~\\
\noindent \fbox{$v_2(b)<0$}
\\~\\
In this case, since $v_2(b)<0$
$$
3v_p(\overline{a}_2)-2v_p(y_1) = v_p(b) = v_p(c) = v_p\left(1-\dfrac{\overline{a}_1\overline{a}_2}{\overline{a}_3}\right) = v_p(3x_1^2+ 10b_2x_1+8\cdot5^2b_4) + v_p(\overline{a}_2) -2v_p(\overline{a}_3).
$$

Hence,
$$
2v_p(\overline{a}_2) = v_p(3x_1^2+ 10b_2x_1+8\cdot5^2b_4)>0.
$$
Note that $v_p(3x_1^2+ 10b_2x_1+8\cdot5^2b_4)>0$ because the point $(x_1,y_1)$ is singular on the curve $E_3$. Therefore,
$v_p(\overline{a}_2) = v_p(-s^2+3x_1+5b_2)>0$, which implies that $v_p(s)\geq 0$. Thus we obtain
$$
v_p(\overline{a}_1),v_p(\overline{a}_2),v_p(\overline{a}_3)\geq 0.
$$
Applying Step $2$ of Tate's algorithm to the curve $E_4$, we obtain
$$
v_p(s^2 + 4\overline{a}_2) >0.
$$todo{¿no sería $(2s)^2 + 4\overline{a}_2$}
Therefore $v_p(s)>0$. Now suppose that $v_p(\overline{a}_2)\geq 2$, then
$$
0> v_p(b) = 3v_p(\overline{a}_2) - 2v_p(y_1)\geq 6 - 2v_p(y_1).
$$
This implies that $v_p(y_1)> 3$. However, this implies that $v_p(\overline{a}_i) \geq i$, which means that we can perform a change of variables that keeps the coefficients of $E_4$ integral at $p$ and decreases the discriminant, contradicting the fact that the equation of $E_4$ is minimal.
\\~\\
Therefore $v_p(\overline{a}_2) = 1$. This implies that
$$
2=2v_p(\overline{a}_2) = v_p(3x_1^2+ 10b_2x_1+8\cdot5^2b_4).
$$

Since $0<v_p(s) = 2 -v_p(y_1)$, we have $0\leq v_p(y_1)\leq 1$, which contradicts
$$
0>v_p(b) = 3- 2v_p(y_1)\geq 1.
$$
\end{proof}

We now proceed to prove Theorem \ref{thm:4} in the case $p=l=5$.

\begin{proposition}\label{prop:10}
Let $E/\mathbb{Q}$ be a rational elliptic curve such that there exists $P\in E(K)[5]\setminus E(\mathbb{Q})$, where $K|\mathbb{Q}$ is a quadratic field. Suppose that $E$ has additive reduction at $5$. Then the following are equivalent
\begin{itemize}
    \item $5| d_K$, where $d_K$ is the discriminant of $K$,
    \item the Kodaira symbol of $E$ at $5$ is neither $II$ nor $III$.
\end{itemize}
Moreover, if $5|d_K$ then the Kodaira symbol of $E$ is $IV^*$ or $III^*$.
\end{proposition}
\begin{proof}
Suppose that $5|d_K$. If we denote $K=\mathbb{Q}\left(\sqrt{d}\right)$ with $d$ a square-free integer, then the twist $E^d$ has a rational point of order $5$. Mentzelos Melistas  proved (\cite[Theorem 1.2]{Melistas2022PurelyAdditive}) that in this case the Kodaira symbol of $E^d$ at $5$ must be $II$ or $III$. Therefore, $E$ has the Kodaira symbol $IV^*$ or $III^*$ at $5$, respectively, since the twist relating the two curves is ramified at $5$. See \cite[Lemma 4.4, Table 2]{Paterson2024MultiquadraticSelmer} for the change of Kodaira Symbol under ramified quadratic twist. 
\\~\\
Conversely, if $5\nmid d_K$, the twist relating $E$ and $E^d$ is unramified at $5$. Therefore the Kodaira symbol at $5$ of both curves is the same. Again, the above result by Mentzelos Melistas  implies that $E^d$ has Kodaira symbol $II$ or $III$.
\end{proof}
Propositions \ref{prop:6}, \ref{prop:9} and \ref{prop:10} along with Theorem \ref{thm:3} prove Theorem \ref{thm:4} for the case $\ell = 5$.

\section{Prime $\ell = 7$}
We follow the same strategy of the case $\ell = 5$. We begin with a pair of lemmas.

\begin{lemma}\label{lemma:5}
Let $E/\mathbb{Q}$ be an elliptic curve with a minimal model at $2$ of the form
$$
E: y^2 =  x^3 +a_2x^2 + a_4x + a_6.
$$
 Suppose that there exists $P\in E(\mathbb{Q}(\sqrt{-7}))[7]\setminus E(\mathbb{Q})$. Then $E$ does not have additive reduction at $2$.
\end{lemma}

\begin{proof}
Arguing by contradiction, suppose that $E$ has additive reduction at $2$. By hypothesis, we have a minimal model at $2$
$$
E: y^2 =  x^3 +a_2x^2 + a_4x + a_6.
$$
 Applying the changes of variables from section \ref{section:2} and following the indications of Remark \ref{rmk:2}, we obtain the twist of $E$ in Tate normal form
$$
\mathcal{T}_{b,c}:  y^2+ (1-c)xy - by = x^3 - bx^2.
$$

Moreover, since the point $(0,0)$ on this curve has order $7$, there exists $t\in\mathbb{Q}$ such that $b = t^2(t-1)$ and $c = t(t-1)$. The relation with the discriminant of the original curve is
\begin{equation}\label{eq:26}
(t-1)^3(t^3-8t^2 + 5t + 1)\overline{a}_3^4= 7^6 t\Delta_1.    
\end{equation}

We divide the proof into cases according to the sign of the valuation of $t$.
\\~\\
\noindent \fbox{$v_2(t)\geq 0$}
\\~\\
The point $(0,0)$ reduces to a singular point when reducing modulo $2$, since it is a $7$-torsion point and the reduction at $2$ is additive (see Remark \ref{rmk:4}). Therefore the curve $\mathcal{T}_{b,c}$ is centered at a point that reduces to a singular point and the coefficients of this curve are $2$-integral by hypothesis. This allows us to apply Tate's algorithm. 
\\~\\
The first step implies that $ 2v_2(t) + v_2(t-1)= v_2(b)>0$,  and therefore, $ v_2(t) + v_2(t-1)= v_2(c)>0$. Now, since the reduction is additive at $2$, the second step implies that
$$
v_2((1-c)^2-4b)>0,
$$
which is a contradiction because $v_2(1-c)= 0$ and $v_2(b)>0$.
\\~\\
\noindent \fbox{$v_2(t)< 0$}
\\~\\
In this case, taking the $2$-adic valuation in the equation \eqref{eq:26} we obtain
\begin{equation}\label{eq:27}
    4 + 4v_2(y_1) = v_2(\Delta_1) - 5v_2(t),
\end{equation}
where we used that $\overline{a}_3=2y_1$. In particular, this equation implies that $v_2(y_1)>0$. 
\\~\\
We have two other equations given by the definitions of $b$ and $c$. On the one hand,
\begin{equation}\label{eq:28}
    3v_2(t) = 3v_2(\overline{a}_2)- 2(v_2(y_1)+1),
\end{equation}
and on the other hand, since $v_2(c)= v_2\left(1-\dfrac{\overline{a}_1\overline{a}_2}{\overline{a}_3}\right)<0$, we have
\begin{equation}\label{eq:29}
    2v_2(t) = v_2(\overline{a}_1) + v_2(\overline{a}_2)- v_2(\overline{a}_3).
\end{equation}
The rest of the proof consists of using these equations, together with the definitions of the coefficients of the curves from Section \ref{section:3}, to obtain a contradiction.
\\~\\
We begin by proving that the valuation of $s$ is non-negative (see the definition of $s$ in Remark \ref{rmk:2}). If $v_2(s)<0$, the definitions of $\overline{a}_i$ for $i\in\{1,2,3\}$ give
$$
v_2(\overline{a}_1) = 1+ v_2(s), \quad v_2(\overline{a}_2) = 2v_2(s), \quad v_2(\overline{a}_3) = 1+ v_2(y_1).
$$
Substituting into equations \eqref{eq:28} and (\ref{eq:29}) we obtain
$$
\left\{
\begin{aligned}
 3v_2(t) &= 2(3v_2(s)- (v_2(y_1)+1))\\
 2v_2(t) &= 1+ (3v_2(s)-(v_2(y_1)+1))
\end{aligned}
\right.
$$
Multiplying the second equation by $-2$ and adding, we obtain $0>v_2(t) = 2$, a contradiction. Therefore we have proved that $v_2(s)\geq 0$, which implies that
$$
\left\{
\begin{aligned}
 v_2(\overline{a}_1) &= 1+ v_2(s)\geq 1,\\
 v_2(\overline{a}_2) &= v_2(-s^2+ 3x_1 -7a_2) \geq 0,\\
 v_2(\overline{a}_3) &= 1+ v_2(y_1) \geq 2.
\end{aligned}
\right.
$$

If $v_2(y_1) = 1$, equation \eqref{eq:27} gives
$$
8 = v_2(\Delta_1) - 5v_2(t).
$$

Hence $v_2(\Delta_1)= 3$ and $v_2(t) = -1$. Substituting into equation (\ref{eq:28}) we obtain
$$
1 = 3v_2(\overline{a}_2),
$$
which is a contradiction. Therefore $v_2(y_1)>1$. Hence
$$
\left\{
\begin{aligned}
 v_2(\overline{a}_1) &= 1+ v_2(s)\geq 1,\\
 v_2(\overline{a}_2) &= v_2(-s^2+ 3x_1 -7a_2) \geq 0,\\
 v_2(\overline{a}_3) &= 1+ v_2(y_1) \geq 3.
\end{aligned}
\right.
$$
Since $E_4$ is a minimal model, $v_2(\overline{a}_2)\in\{0, 1\}$. Subtracting equation \eqref{eq:28} from \eqref{eq:29} and substituting the definition of $s$ (see Remark \ref{rmk:2}), we obtain
\begin{equation}\label{eq:30}
2v_2(\overline{a}_2) = v_2(7^2a_4 - 2\cdot 7a_2x_1 + 3x_1^2)+ 1 + v_2(t).
\end{equation}
If $v_2(\overline{a}_2) = 0$, then
$$
-v_2(t) = v_2(7^2a_4 - 2\cdot 7a_2x_1 + 3x_1^2)+1\geq v_2(y_1)+2.
$$
In the last inequality we used that $v_2(s)\geq 0$. Using this inequality in equation \eqref{eq:27} we obtain
$$
4 = v_2(\Delta_1) - v_2(t) + 4(-v_2(t)- v_2(y_1)) \geq 8.
$$
This is a contradiction. Finally we apply the same strategy with $v_2(\overline{a}_2) = 1$. In this case, equation \eqref{eq:30} gives
$$
-v_2(t) = v_2(7^2a_4 - 2\cdot 7a_2x_1 + 3x_1^2)-1\geq v_2(y_1).
$$
Using this inequality in equation \eqref{eq:27}  we obtain
$$
4 = v_2(\Delta_1) - v_2(t) + 4(-v_2(t)- v_2(y_1))\geq 4.
$$
This implies that $v_2(\Delta_1) = -v_2(t) = v_2(y_1) = 2$ Substituting these values into equation \eqref{eq:28} we obtain
$$
-6 = 1-2v_2(y_1) = -3,
$$
a contradiction.
\end{proof}
\begin{lemma}\label{lemma:6}
Let $E/\mathbb{Q}$ be an elliptic curve. Suppose that there exists $P\in E(\mathbb{Q}(\sqrt{-7}))[7]\setminus E(\mathbb{Q})$. Then $E$ does not have additive reduction at $2$.
\end{lemma}

\begin{proof}
Arguing by contradiction, suppose that $E$ has additive reduction at $2$. We can always find a minimal equation for $E$
$$
E_1 : y^2 + a_1xy + a_3y = x^3 + a_2x^2 + a_4x + a_6
$$
with integer coefficients such that $a_1,a_3\in\{0,1\}$ and $a_2\in\{-1,0,1\}$. After carrying out the procedure explained in Section \ref{section:3}, we obtain the twist of $E_1$ in Tate normal form
$$
\mathcal{T}_{b,c}:  y^2+ (1-c)xy - by = x^3 - bx^2
$$
with $b :=- \dfrac{\overline{a}_2^3}{\overline{a}_3^2}$ and $c:= 1- \dfrac{\overline{a}_1\overline{a}_2}{\overline{a}_3}$. 
\\~\\
Moreover, since the point $(0,0)$ on this curve has order $7$, there exists $t\in\mathbb{Q}$ such that $b = t^2(t-1)$ and $c = t(t-1)$. The relation with the discriminant of the original curve is
$$
(t-1)^3(t^3-8t^2 + 5t + 1)\overline{a}_3^4= 7^6 2^{12}t\Delta_1.
$$
Note that if $v_2(t)\geq 0$, we can argue exactly as in the proof of the previous lemma and reach a contradiction. Therefore $v_2(t)<0$. Taking the $2$-adic valuation in the previous equation and grouping terms, we obtain
$$
4v_2(y_1) = 8 + v_2(\Delta_1) - 5v_2(t).
$$
This implies that $v_2(y_1)\geq 4$. Since the point $(x_1,y_1)$ on the curve $E_3$ has order $7$ and the reduction at $2$ is additive, this point reduces to a singular point modulo $2$ (see Remark \ref{rmk:4}). Differentiating the equation of the reduced curve  $y^2=x^3 - 7b_2x^2 + 8\cdot 7^2b_4 x - 16 \cdot 7^3 b_6$ with respect to $x$, and taking into account that it must vanish at $(x_1, y_1)$, we obtain
$$
3x_1^2\equiv 0 \mod{2}.
$$
Hence $v_2(x_1)>0$. Moreover, since $(x_1,y_1)$ is a point of $E_3$, we have the equation
\begin{equation}\label{eq:31}
    y_1^2= x_1^3 -7b_2x_1^2+ 8\cdot7^2b_4x_1  -16\cdot 7^3b_6
\end{equation}
We divide the proof according to the valuation of $b_2$.
\\~\\
\noindent \fbox{$v_2(b_2)> 0$}
\\~\\
The definition $b_2 = a_1^2+ 4a_2$, together with the hypothesis that $a_1\in \{0,1\}$ and $a_2\in\{-1,0,1\}$, implies that $v_2(b_2) \geq 2$ and $a_1 = 0$. The latter also implies that $b_4=  2a_4 + a_1a_3 = 2a_4$, hence $v_2(b_4) > 0$. Taking the $2$-adic valuation in equation (\ref{eq:31}), we obtain
$$
8\leq v_2(x_1^3-7b_2x_1^2 + 8\cdot 7^2b_4x_1- 16\cdot 7^3b_6).
$$
The valuation of each term to the right of $x_1^3$ is greater than or equal to $4$, hence $v_2(x_1)>1$. This implies that the valuation of each term to the left of $16\cdot 7^3b_6$ is greater than or equal to $6$. Therefore $v_2(b_6)>0$. The definition of $b_6$ is $b_6 = a_3^2+4a_6$. Hence $a_3 = 0$. We have reached that in the original equation $a_1 = a_3 = 0$. Therefore, by the previous lemma we obtain a contradiction.
\\~\\
\noindent \fbox{$v_2(b_2)=0$}
\\~\\
In this case, since the curve has additive reduction at $2$, we have $v_2(b_2^2-24b_4) = v_2(c_4)>0$, which contradicts $v_2(b_2)=0$.
\end{proof}
\begin{proposition}\label{prop:11}
Let $E/\mathbb{Q}$ be an elliptic curve and let $p\neq 7$ be a prime. Suppose that there exists $P\in E(\mathbb{Q}(\sqrt{-7}))[7]\setminus E(\mathbb{Q})$. Then $E$ does not have additive reduction at $p$.
\end{proposition}
\begin{proof}
First, by Lemma \ref{lemma:6}, we may assume that $p\neq 2$. We begin with the curve $E$ with a minimal model
$$
E_1 : y^2 + a_1xy + a_3y = x^3 + a_2x^2 + a_4x + a_6
$$
with integer coefficients such that $a_1,a_3\in\{0,1\}$ and $a_2\in\{-1,0,1\}$. Arguing by contradiction, suppose that $E$ has additive reduction at $p$. Following the procedure in Section \ref{section:3}, we obtain the twist of $E_1$ over $\mathbb{Q}\left(\sqrt{-7}\right)$ in Tate normal form
$$
\mathcal{T}_{b,c}:  y^2+ (1-c)xy - by = x^3 - bx^2.
$$
Moreover, since the point $(0,0)$ on this curve has order $7$, there exists $t\in\mathbb{Q}$ such that $b = t^2(t-1)$ and $c = t(t-1)$. The relation with the discriminant of the original curve is
$$
(t-1)^3(t^3-8t^2 + 5t + 1)\overline{a}_3^4= 2^{12}7^6 t\Delta_1.
$$
Note that if $v_2(t)\geq 0$, we can argue exactly as in the proof of Lemma \ref{lemma:5} and reach a contradiction. Therefore $v_2(t)<0$. Taking the $p$-adic valuation in the previous equation we obtain
\begin{equation}\label{eq:32}
    4v_p(y_1) = v_p(\Delta_1) - 5v_p(t),
\end{equation}
which implies that $v_p(y_1)\geq 2$.
\\~\\
The definitions of $b$ and $c$ are $b :=- \dfrac{\overline{a}_2^3}{\overline{a}_3^2}$ and $c:= 1- \dfrac{\overline{a}_1\overline{a}_2}{\overline{a}_3}$. Taking the $p$-adic valuation we obtain, on the one hand,
\begin{equation}\label{eq:33}
    3v_p(t) = 3v_p(\overline{a}_2)- 2v_p(y_1),
\end{equation}
and on the other hand, since $v_p(c)= v_p\left(1-\dfrac{\overline{a}_1\overline{a}_2}{\overline{a}_3}\right)<0$, we have
\begin{equation}\label{eq:34}
    2v_p(t) = v_p(\overline{a}_1) + v_p(\overline{a}_2)- v_p(\overline{a}_3).
\end{equation}
We have $v_p(\overline{a}_1) = v_p(s)$ and if $v_p(s)<0$, then $v_p(\overline{a}_2) = 2v_p(s)$. Substituting into equations \eqref{eq:33} and \eqref{eq:34}, we obtain
$$
\left\{
\begin{aligned}
 3v_p(t) &= 2(3v_p(s)- v_p(y_1))\\
 2v_p(t) &= (3v_p(s)-v_p(y_1))
\end{aligned}
\right.
$$
Hence $v_p(t) = 0$, a contradiction. Therefore we have proved that $v_p(s)\geq 0$. This implies that $v_p(\overline{a}_2) = v_p(-s^2+3x_1-7b_2)\geq 0$.
\\~\\
Moreover, if $v_p(\overline{a}_2)= 0$, subtracting equations \eqref{eq:33} and \eqref{eq:34} we obtain
$$
-v_p(t) = v_p(8\cdot 7^2b_4 - 2 \cdot 7b_2x_1 + 3x_1^2)\geq v_p(y_1),
$$
since $v_p(s)\geq 0$. Applying this inequality in equation \eqref{eq:32}, we obtain
$$
0 = v_p(\Delta_1) - v_p(t) - 4(v_p(t) + v_p(y_1))>0,
$$
a contradiction. Therefore $v_p(\overline{a}_2)>0$.
\\~\\
The curve of model $E_4$ has additive reduction at $p$ because of Remark \ref{rmk:3}. We can apply Tate's algorithm to the curve $E_4$, since all its coefficients are $p$-integral and it is centered at a point of the singular fiber by Remark \ref{rmk:4}. Step $2$ of Tate's algorithm shows that  we have
$$
v_p(\overline{a}_1^2+4\overline{a}_2)>0.
$$
This implies that $v_p(s)>0$. By Remark \ref{rmk:3}, the equation of $E_4$ is minimal. Note that if $v_p(y_1) = 2$, then equation \eqref{eq:33} implies that $3v_p(t)=3v_p(\overline{a}_2) - 4$, which is a contradiction because $v_p(t)<0$ and $v_p(\overline{a}_2)>0$. Therefore $v_p(\overline{a}_3) = v_p(y_1)\geq 3$, $v_p(\overline{a}_1)\geq 1$ and $v_p(\overline{a}_2)\geq 1$. By the minimality of $E_4$, necessarily $v_p(\overline{a}_2) = 1$.
\\
Subtracting equations \eqref{eq:34} and \eqref{eq:33} we obtain
$$
-v_p(t) = v_p(8\cdot 7^2b_4 - 2 \cdot 7b_2x_1 + 3x_1^2)- 2\geq v_p(y_1)-1,
$$
where we used that $v_p(s)>0$.
\\~\\
Using this inequality in equation \eqref{eq:32} we obtain
$$
0 = v_p(\Delta_1) - 5v_p(t) - 4v_p(y_1) \geq v_p(\Delta_1) + v_p(y_1) - 5.
$$

Since $v_p(y_1)\geq 3$ and the reduction is additive at $p$, we obtain that $v_p(y_1) = 3$ and $v_p(\Delta_1) = 2$. Hence, substituting in the previous equation, we get $v_p(t)= -2$.
\\~\\
Substituting the values of each term into equation \eqref{eq:33}, we obtain
$$
-6 = 3- 6= -3,
$$
a contradiction.
\end{proof}
\begin{proposition}\label{prop:12}
    Let $E/\mathbb{Q}$ be a rational elliptic curve such that there exists $P\in E(K)[7]\setminus E(\mathbb{Q})$, where $K/\mathbb{Q}$ is a quadratic field. Suppose that $E$ has additive reduction at $7$. Then the following are equivalent:
    \begin{itemize}
        \item $7 \mid d_K$, where $d_K$ is the discriminant of $K$.
        \item The Kodaira symbol of $E$ at $7$ is not $II$.
    \end{itemize}
    Moreover, if $7 \mid d_K$, then $E$ has Kodaira symbol $IV^*$.
\end{proposition}
\begin{proof}
Suppose that $7|d_K$. If we denote $K = \mathbb{Q}\left( \sqrt{d}\right)$ with $d$ a square-free integer, then the twist $E^d$ has a rational point of order $7$. Mentzelos Melistas   proved (\cite[Theorem 1.2]{Melistas2022PurelyAdditive}) that in this case the Kodaira symbol of $E^d$ at $7$ must be $II$. Therefore, $E$ has Kodaira symbol $II^*$ at $7$, since the twist relating the two curves is ramified at $7$.
\\~\\
Conversely, if $7\nmid d_K$, the twist relating $E$ and $E^d$ is unramified at $7$. Therefore, the Kodaira symbol at $7$ of both curves is the same. Again, the result of Mentzelos Melistas  implies that $E^d$ has Kodaira symbol $II$.
\end{proof}
Propositions \ref{prop:8}, \ref{prop:11} and \ref{prop:12} along with Theorem \ref{thm:3} prove Theorem \ref{thm:4} in the case of $\ell = 7$. This concludes the proof of Theorem \ref{thm:4}.
\section{Conclusion}
All the results of this article are summarized in Theorem \ref{thm:4} and Theorem \ref{thm:5}. Going back to our original problem 
\addtocounter{problem}{-1}
\begin{problem}
Is there a precise (and easy) description of which are the possible quadratic extensions
$K|\mathbb{Q}$ with $E_{\tors}(\mathbb{Q}) \neq E_{\tors}(K)$ , ideally in terms of some invariant(s) of the curve?
\end{problem}
The results of this article bring us closer to a satisfactory answer to that question. Indeed, we understand almost completely the ramification of $K$. Note that in the case of quadratic fields knowing the ramification of $K$ and knowing $K$ is the same. 
\\~\\
Concretely, if there's a point $P\in E(K)[\ell]\setminus E(\mathbb{Q})$, then $\ell\in\{2,3,5,7\}$ and we have:
\begin{itemize}
    \item If $\ell=2$, then $K = \mathbb{Q}(\sqrt{\Delta})$, where $\Delta$ is the minimal discriminant of $E$.
    \item If $\ell = 5,7$, then we can determine all the primes that ramify at $K$. They are all $p\neq\ell$ such that $E$ has additive reduction along with maybe $\ell$ depending on the Kodaira symbol at $\ell$. To determine $K$ it only remains to know the power of $2$ at $d_K$ and its sign.
    \item If $\ell = 3$, then we can determine the primes $p\neq 3$ that ramify on $K$ looking at the Kodaira symbols at the same prime. In this case, $K$ is not necessary unique, but there can be at most $2$. In case $K$ is not unique we can say more about the ramification of $3$ (see Remark \ref{rmk:5}). It remains to know the power of $2$ that divides $d_k$ along with his sign and the ramification of $3$. Note that the ramification of $3$ has been studied in this article in the case when $E(K)\cong \mathcal{C}_9$.
\end{itemize}
It also remains to study the mixed case defined in the introduction. Future work by the author is already in progress concerning these matters.
\section{Appendix: Tate's Algorithm}

The purpose of this appendix is to facilitate the reading of the article by collecting, in one place, the steps of Tate's algorithm as presented in \cite[p.364]{Silverman1994AdvancedTI} in order to spare the reader from having to consult \cite{Silverman1994AdvancedTI}. We do not introduce any new results or modifications to the algorithm here. Rather, we simply reproduce the relevant steps in a form convenient for reference.
\\~\\
Let $E/K$ be an elliptic curve defined over a local field and $\pi$ an uniformizer for the rings of integers $R$ of $K$ with residue field $k$. We will follow the description of Tate's algorithm of \cite[p.364]{Silverman1994AdvancedTI} skipping the computation of some quantities that we are not interested in.  As we go along the algorithm we will be making various assumptions. These assumptions are cumulative, and will be \fbox{boxed} for clarity.
    \\~\\
     Consider a model for $E$
    \begin{equation}\label{eq:35}
       E: y^2+a_1xy + a_3y = x^3+a_2x^2+a_4x+a_6
    \end{equation}
    where $a_i\in R$. During the algorithm, we are going to transform the Weierstrass equation so as to make the $a_i$'s more and more divisible by $\pi$. To keep track we introduce the notation
    \begin{equation*}
             a_{i,r} = \pi^{-r}a_i.
     \end{equation*}
     The following quantities:
    \begin{itemize}
        \item Kodaira symbol and reduction type.
        \item $v(\mathcal{D}_E/K)$: The valuation of the minimal discriminant.
        \item $m(E/K)$: The number of components, defined over $\overline{k}$ and counted without multiplicity, on the special fiber of a minimal proper regular model of $E$ over $R$. 
        \item The exponent $f(E/K)$ of the conductor.
    \end{itemize}
    can be computed executing the following algorithm. 
    \\~\\
    We start with the model of \eqref{eq:35}. It is not required to be minimal, nevertheless, when the algorithm finishes, the resulting model will be a minimal model and its discriminant will have valuation $v(\mathcal{D}_E/K)$ .
    \begin{itemize}
        \item \textit{Step 1}. If $\pi\nmid\Delta$, then the special fiber of a minimal proper regular model of $E$ over $R$ is an elliptic curve and we have 
        $$
        \text{Kodaira Symbol: $I_0$}, \quad v(\Delta) = 0, \quad m=1, \quad f=0.
        $$
        \item \textit{Step 2}. Assume \fbox{$\pi|\Delta$}, this means that $\tilde{E}$ has a singular point. Make a change of variables to move the singular point to $(0,0)$. Then \fbox{$\pi|a_3, a_4 \text{ and } a_6$}. If $\pi\nmid b_2$, then the reduction is multiplicative and the Kodaira symbol is $I_n$ where $n = v(\Delta)\geq 1$, $m=n$ and $f=1$.
        \item \textit{Step 3}. Assume now that \fbox{$\pi|b_2$}. If $\pi^2\nmid a_6$, then 
         $$
        \text{Kodaira Symbol: $II$}, \quad f = v(\Delta), \quad m=1.
        $$
        \item \textit{Step 4}. Assume that \fbox{$\pi^2|a_6$} (which implies that $\pi^2|b_6$ and $\pi^2|b_8$). If $\pi^3\nmid b_8$, then 
        $$
        \text{Kodaira Symbol: $III$}, \quad f = v(\Delta)-1, \quad m=2.
        $$
        \item \textit{Step 5}. Assume that \fbox{$\pi^3|b_8$} (which implies that $\pi^2|b_4$). If $\pi^3\nmid b_6$, then 
        $$
        \text{Kodaira Symbol: $IV$}, \quad f = v(\Delta)-2, \quad m=3.
        $$
        \item \textit{Step 6}. Assume that \fbox{$\pi^3|b_6$}. Then we can change the coordinates to get
        $$
        \fbox{$\pi|a_1,a_2$, \quad $\pi^2|a_3,a_4$, \quad $\pi^3|a_6$}.
        $$
        More precisely, the boxed assumptions up to this point show that we can factor
        \begin{eqnarray*}
            Y^2+a_1Y-a_2 &\equiv& (Y-\alpha)^2 \pmod \pi\\
            Y^2+a_{3,1}Y - a_{6,2} &\equiv& (Y-\beta)^2 \pmod \pi 
        \end{eqnarray*}
        and then the substitution $y' = y + \alpha x + \beta \pi$ will have the desired effect. After doing this, we consider the polynomial 
        $$
        P(T) = T^3+a_{2,1}T^2+a_{4,2}T + a_{6,3}.
        $$
        If $P(T)$ has distinct roots in $\overline{k}$ (i.e., if $\pi$ does not divide the discriminant of $P$), then 
          $$
        \text{Kodaira Symbol: $I_0^*$}, \quad f = v(\Delta)-4, \quad m=5.
        $$
        \item \textit{Step 7}.  If $P(T)$ has a simple root and one double root in $\overline{k}$, then 
         $$
        \text{Kodaira Symbol: $I_n^*$}, \quad f = v(\Delta)-4-n,  \quad m=n+5.
        $$
        If $p\neq 2$, then $n = v(\Delta) - 6$, so $f = 2$ and $m = v(\Delta)-1$. For arbitrary $p$, $n$ can be computed using a subprocedure that we won't require (see \cite[p. 367]{Silverman1994AdvancedTI}).
        \item \textit{Step 8}. Suppose now that $P(T)$ has a triple root in $\overline{k}$. Making a translation on $x$, we may assume that the root is $T=0$, which means that \fbox{$\pi^2|a_2,\pi^3|a_4$ and  $\pi^4|a_6$}. If the polynomial $Y^2+a_{3,2}Y-a_{6,4}$ has distinct roots in $\overline{k}$. Then, 
        $$
        \text{Kodaira Symbol: $IV^*$}, \quad f = v(\Delta)-6,  \quad m=7.
        $$
        \item \textit{Step 9}. Suppose that $Y^2+a_{3,2}Y-a_{6,4}$ has a double root in $\overline{k}$. Making a translation on $y$, we may assume that the root is $Y=0$, which means that \fbox{$\pi^3|a_3$ and $\pi^5|a_6$}. If $\pi^4\nmid a_4$, then, 
        $$
        \text{Kodaira Symbol: $III^*$}, \quad f = v(\Delta)-7,  \quad m=8.
        $$
        \item \textit{Step 10}. Suppose that \fbox{$\pi|a_4$}. If $\pi^6\nmid a_6$, then 
        $$
        \text{Kodaira Symbol: $II^*$}, \quad f = v(\Delta)-8,  \quad m=9.
        $$
        \item \textit{Step 11}. Finally, suppose that $\pi^6|a_6$. Then the original Weierstrass equation was not minimal. The substitution $(x,y) = (\pi^2x',\pi^3y')$ leads to the equation
        $$
        y'^2 + a_{1,1}x'y' + a_{3,3}y' = x'^3+a_{2,2}x'^2+a_{4,4}x' +a_{6,6}
        $$
        with coefficients in $R$ and discriminant $\Delta'=\pi^{-12}\Delta$. Go back to Step 1 and begin the algorithm again with this new equation. Note that we can only get to Step 11 a finite number of times, since each time we reach it, the valuation of the discriminant decreases by $12$.
    \end{itemize}
\bibliographystyle{amsplain}
\bibliography{QuadTor}
\end{document}